\documentclass{amsart}
\usepackage{amssymb,latexsym}
\usepackage{amscd,amsthm}
\usepackage{tikz-cd}

\usepackage[all]{xy}

\newtheorem{theorem}{Theorem}[section]
\newtheorem{lemma}[theorem]{Lemma}
\newtheorem{proposition}[theorem]{Proposition}
\newtheorem{corollary}[theorem]{Corollary}

\theoremstyle{definition}
\newtheorem{definition}[theorem]{Definition}

\newtheorem{remark}{Remark}

\DeclareMathOperator{\Ext}{Ext}
\DeclareMathOperator{\Hom}{Hom}

\DeclareMathOperator{\cok}{cok}
\DeclareMathOperator{\im}{Im}

\newcommand{\cat}[1]{\mathcal{#1}}           %% font for categories

\newcommand{\class}[1]{\mathcal{#1}}   %% font for classes

\newcommand{\Z}{\mathbb{Z}}

\newcommand{\mathcolon}{\colon\,} %% Hovey uses for maps, like f: A -> B

\newcommand{\ch}{\textnormal{Ch}(R)}
\newcommand{\chain}[1]{\textnormal{Ch}(#1)}
\newcommand{\cha}[1]{\textnormal{Ch}(\mathcal{#1})}

\newcommand{\dwclass}[1]{dw\widetilde{\class{#1}}}
\newcommand{\exclass}[1]{ex\widetilde{\class{#1}}}

\newcommand{\tilclass}[1]{\widetilde{\class{#1}}}
\newcommand{\dgclass}[1]{dg\widetilde{\class{#1}}}

\newcommand{\rightperp}[1]{#1^{\perp}}
\newcommand{\leftperp}[1]{{}^\perp #1}

\newcommand{\homcomplex}{\mathit{Hom}}

\newcommand{\Ho}[1]{{\textnormal{Ho}(#1)}}

\begin{document}

\title{Exact model structures and recollements}

\author{James Gillespie}
\address{Ramapo College of New Jersey \\
         School of Theoretical and Applied Science \\
         505 Ramapo Valley Road \\
         Mahwah, NJ 07430}
\email[Jim Gillespie]{jgillesp@ramapo.edu}
\urladdr{http://pages.ramapo.edu/~jgillesp/}

\dedicatory{Dedicated to David Gillespie on the occasion of his $0^{th}$ birthday}

\date{\today}

\begin{abstract}%% what do you actually prove?%%
We show how to obtain recollements of triangulated categories using the theory of exact model structures from~\cite{gillespie-exact model structures}. After noting how the theory relates to well-known notions in the simplest case of Frobenius categories, we apply these ideas to recollements of homotopy categories of chain complexes. In short, we give model categorical explanations for the classical Verdier localization recollement as well as several recollement situations due to Neeman and Murfet.

\end{abstract}

\maketitle

\section{introduction}

This work is a continuation of~\cite{gillespie-recollements}, where the author showed that there is a strong link between recollements of triangulated categories and the cotorsion pairs which serve as abelian model structures for those categories. But the abelian setting is too restrictive for many interesting applications, and so the idea in this paper is that we should relax the hypotheses in the theory from the abelian setting to the more general setting of exact categories, in the sense of Quillen~\cite{quillen-algebraic K-theory}. We offer an alternate proof and statement of the main theorem from~\cite{gillespie-recollements} yielding a recollement from three cotorsion pairs and we give applications showing how some of the most common and interesting recollements in algebraic geometry can be obtained using exact model structures.

As a motivating example, lets consider a well-known recollement. Let $R$ be a ring and denote by $K(R)$ its homotopy category. So $K(R)$ is the category of chain complexes of $R$-modules with morphisms the homotopy classes of chain maps. Recall that the derived category $\class{D}(R)$ is, by definition, the category obtained from $K(R)$ by formally inverting the homology isomorphisms. This is often written as $\class{D}(R) = K(R)[\class{W}^{-1}]$ where $\class{W}$ is the class of all homology isomorphisms. $K(R)$ and $\class{D}(R)$ are each triangulated categories and the kernel of the quotient functor $K(R) \xrightarrow{Q} \class{D}(R)$ is precisely the full subcategory $\class{E}/\hspace{-.04in}\sim$ of all exact complexes (modulo the chain homotopy relation $\sim$). The functor $Q$ has both a left and a right adjoint. The right adjoint takes a complex to its K-injective resolution while its left adjoint takes a complex to its K-projective resolution. All told, the localization $\class{E}/\hspace{-.04in}\sim \ \xrightarrow{} K(R) \xrightarrow{Q} \class{D}(R)$ forms the center arrows in a recollement diagram
\[
\xy
(-27,0)*+{\class{E}/\hspace{-.04in}\sim};
(0,0)*+{K(R)};
(27,0)*+{\class{D}(R)};
{(-19,0) \ar (-10,0)};
{(-10,0) \ar@<0.4em> (-19,0)};
{(-10,0) \ar@<-0.4em> (-19,0)};
{(10,0) \ar (19,0)};
{(19,0) \ar@<0.4em> (10,0)};
{(19,0) \ar@<-0.4em> (10,0)};
\endxy
.\]
Although this diagram holds a lot of information, Corollary~\ref{cor-Verdier recollement via cotorsion pairs} gives a model categorical description for \emph{all} of this information at once. It can be explained quite easily as follows in terms of cotorsion pairs in the category of chain complexes together with the degreewise split short exact sequences: Let $\class{W}$ denote the class of all contractible complexes, $\class{E}$ denote the class of all exact complexes, $K\class{I}$ the class of all K-injective complexes, $K\class{P}$ the class of all K-projective complexes, and $\class{A}$ the class of all complexes. Then we have three complete cotorsion pairs which are equivalent to Quillen model structures for the homotopy categories as follows.
\begin{itemize}
\item $(\class{W}_1,\class{F}_1) = (\class{W},\class{A})$ = A Quillen model structure for $K(R)$.
\item $(\class{W}_2,\class{F}_2) = (K\class{P},\class{E})$ = A Quillen model structure for $\class{E}/\hspace{-.04in}\sim$.
\item $(\class{W}_3,\class{F}_3) = (\class{E},K\class{I})$ = A Quillen model structure for $\class{D}(R)$.
\end{itemize}
Having these three cotorsion pairs tells us at once that $K(R)$, $\class{E}/\hspace{-.04in}\sim$, and $\class{D}(R)$ are triangulated categories. Moreover, the recollement diagram is automatic from the fact that $\class{W}_3 \cap \class{F}_1 = \class{F}_2$. Indeed Theorem~\ref{them-recollements in krause form} says that this happens in far more general situations when we have three nice cotorsion pairs in an exact category which is weakly idempotent complete. All weakly idempotent complete means is that every split monomorphism has a cokernel.

For convenience we will refer to weakly idempotent complete categories possessing a Quillen exact structure as \emph{WIC exact categories} in this paper. The author showed in~\cite{gillespie-exact model structures} that the correspondence between model structures and cotorsion pairs from~\cite{hovey} carries over to the case of WIC exact categories. So to be more precise, Theorem~\ref{them-recollements in krause form} says the following. Suppose we have three injective cotorsion pairs $\class{M}_1 = (\class{W}_1, \class{F}_1)$, $\class{M}_2 = (\class{W}_2, \class{F}_2)$, and $\class{M}_3 = (\class{W}_3, \class{F}_3)$ in a WIC exact category $\cat{A}$ with enough injectives. Then each corresponds to an injective model structure on $\cat{A}$ having $\class{F}_i$ as its class of fibrant objects and having a triangulated homotopy category $\cat{A}/\class{W}_i \,\cong\, \class{F}_i/\sim \,$. It turns out here that $\class{F}_i$ is a Frobenius category and the formal $\sim$ is characterized by the expected: $f \sim g$ iff $g-f$ factors through an injective. If the three cotorsion pairs satisfy the simple containments $\class{F}_3 \subseteq \class{F}_1$ and $\class{W}_3 \cap \class{F}_1 = \class{F}_2$, then we get a recollement diagram. See Theorem~\ref{them-recollements in krause form} where there is also a picture of the recollement and a description of the involved functors. Theorem~\ref{them-recollements theorem projective version} states the projective dual. The benefit the author sees in this, is that it allows us to take an enormous amount of technical attention away from the actual categories and adjoints, and to focus on the much simpler cotorsion pairs themselves. See also the final Remark of the paper.

So while the first purpose of this paper is to show that the notion of a WIC exact model structure provides a natural and convenient formal language to easily discuss recollements, the second purpose of this paper is to give examples. The details to the motivating example above appear in Section~\ref{sec-K-injective and K-projective model structures}. But it is in Section~\ref{sec-Murfet and Neeman models and recollements} where the most interesting applications appear. Here we find the correct model structures for which the recollements due to Neeman and Murfet from~\cite{neeman}, \cite{murfet}, and~\cite{neeman-adjoints} immediately follow. That is, we construct exact model structures, both an injective version and in the affine case a projective version, for which the recollements follow at once from Theorems~\ref{them-recollements in krause form} and~\ref{them-recollements theorem projective version}. In particular, given a ring $R$, we will describe explicitly two balanced model structures, one injective and one projective, for the category $\chain{\class{F}}/\tilclass{F}$ = The derived category of the exact category of flat modules. An interesting aspect of the approach is that all of the model structures appearing in this section are obtained by simply restricting known model structures on $\ch$, to the category $\chain{\class{F}}$ of chain complexes of flat modules.

Some of what we do relies on working with the category of chain complexes along with the degreewise split exact structure. Such a category is always a Frobenius category as we recall in Section~\ref{sec-preliminaries}. So when beginning the applications in Section~\ref{sec-localizing cot pairs in frobenius categories} we start by seeing what Theorem~\ref{them-recollements in krause form} says in the easiest setting, that of a Frobenius category. In this setting the notion of an injective cotorsion pair coincides with the notion of a projective cotorsion pair, and we call them \emph{localizing cotorsion pairs}. We conclude that in the Frobenius case, Theorem~\ref{them-recollements in krause form} recovers the well-known idea of a torsion triple. The only difference is that we describe things as \emph{localizing cotorsion triples} in the Frobenius category $\cat{A}$, rather than the torsion triples which exist on the level of the stable category $\cat{A}/\sim$. But they are the same thing. For example, in the illustrative example we gave above $(K\class{P},\class{E},K\class{I})$ is a localizing cotorsion triple and yields the recollement.

We also wish to point out that our work clearly relates to that of Saor{\'{\i}}n and Jan {\v{S}}t'ov{\'{\i}}{\v{c}}ek in~\cite{saorin-stovicek}. Here the authors go into more detail on how the small object argument can be adapted to the exact category setting. This again points to the idea that exact categories work as a nice setting for much of the modern trends in homological algebra.

\section{preliminaries: WIC exact model structures and chain complexes}\label{sec-preliminaries}

This paper is about recollements of triangulated categories but the reader really just needs to know about cotorsion pairs or model structures.  We point out that the paper relies on some work in~\cite[Sections~3 and~4]{gillespie-recollements} and~\cite[Sections~1--4]{gillespie-exact model structures}. In fact, this paper is a continuation of~\cite{gillespie-recollements}, but applying the theory to more general situations using the ideas in~\cite{gillespie-exact model structures}. The definition of a recollement that we use is the same as in~\cite[Sections~3 and~4]{gillespie-recollements} and this is the standard definition. The author learned about recollements from~\cite{krause-stable derived cat of a Noetherian scheme} and there is a more thorough treatment in~\cite{krause-localization theory for triangulated categories}. The standard reference is~\cite{BBD-perverse sheaves}. The related notion of a torsion triple in a triangulated category is briefly discussed in~\cite{gillespie-recollements} but more information can be found in~\cite{beligiannis-reiten}.

The paper~\cite{gillespie-exact model structures} is brief and shows a few things which we will summarize now in more detail.
Recall that an \emph{exact category} in the sense of~\cite{quillen-algebraic K-theory} is a pair $(\class{A},\class{E})$ where $\class{A}$ is an additive category and $\class{E}$ is a class of short exact sequences. Here, a short exact sequence is an actual kernel-cokernel pair  $A \rightarrowtail B \twoheadrightarrow C$, but in this context it is only called a short exact sequence if it actually lies in the specified class $\class{E}$. For this reason, it is better to call such a sequence an \emph{admissible short exact sequence}, and to call  $A \rightarrowtail B$ (resp. $B \twoheadrightarrow C$) an \emph{admissible monomorphism} (resp. \emph{admissible epimorphism}). Many authors use the alternate terms \emph{conflation}, \emph{inflation}, and \emph{deflation}. The class $\class{E}$ must satisfy a few axioms which allow for some basic constructions. For example, all of the split exact sequences  $A \rightarrowtail A \oplus B \twoheadrightarrow B$ are assumed to be in $\class{E}$. Also the class is closed under isomorphisms, pushouts along admissible monomorphisms, and pullbacks along admissible epimorphisms. Finally, the admissible monomorphisms (resp. admissible epimorphisms) are closed under composition. A fundamental result of all this is that the usual Yoneda Ext bifunctor $\Ext^1_{\cat{A}}(A,B)$ construction will hold. Exact categories were introduced by Quillen in~\cite{quillen-algebraic K-theory}. The author highly recommends~\cite{buhler-exact categories} to the interested reader.

Now we briefly summarize a couple things from~\cite{gillespie-exact model structures}. First, in an attempt to generalize Hovey's 1-1 correspondence between abelian model structures and cotorsion pairs, the author realized that to get a closed model structure we need a small assumption on the additive category $\cat{A}$. Fortunately, this concept, of $\cat{A}$ being \emph{weakly idempotent complete}, was already explained nicely in~\cite{buhler-exact categories}. The idea is simple. Recall that a monomorphism $f : A \xrightarrow{} B$ is \emph{split} if there exists a $g : B \xrightarrow{} A$ such that $gf=1_A$. We expect such a monomorphism to actually split! But we don't get the decomposition $B \cong A \oplus \cok{f}$ unless $\cok{f}$ exists. An additive category is called \emph{weakly idempotent complete} if all split monomorphisms have a cokernel, or equivalently, all split epimorphisms have a kernel. For notational convenience we set the following language that we use in this paper.

\begin{definition}\label{def-WIC exact category}
Let $\cat{A} = (\cat{A},\class{E})$ be an exact category. If the underlying additive category is weakly idempotent complete, then we also say the exact category $\cat{A}$ is a weakly idempotent complete. For short, we will call this a \emph{WIC exact} category.
\end{definition}

We point out that for a weakly idempotent complete category $\cat{A}$, the notion of retracts coincides with direct summands. Also, a full subcategory $\cat{S}$ of an abelian category naturally inherits the structure of an exact category whenever $\cat{S}$ is closed under extensions. If $\cat{S}$ is also closed under direct summands then this inherited exact structure is WIC.  With this notion of a WIC exact category we get the following analog of Hovey's one-to-one correspondence as below, which appeared in~\cite[Corollary~3.4]{gillespie-exact model structures}.

\begin{theorem}\label{them-Hovey's theorem for WIC-exact categories}
Let $\class{A}$ be a WIC exact category. Then there is a one-to-one correspondence between exact model structures on $\cat{A}$ and complete cotorsion pairs $(\class{Q},\class{R} \cap \class{W})$ and $(\class{Q} \cap \class{W} , \class{R})$ where $\class{W}$ is a thick subcategory of $\cat{A}$. Given a model structure, $\class{Q}$ is the class of cofibrant objects, $\class{R}$ the class of fibrant objects and $\class{W}$ the class of trivial objects. Conversely, given the cotorsion pairs with $\class{W}$ thick, a cofibration (resp. trivial cofibration) is an admissible monomorphism with a cokernel in $\class{Q}$ (resp. $\class{Q} \cap \class{W}$), and a fibration (resp. trivial fibration) is an admissible epimorphism with a kernel in $\class{R}$ (resp. $\class{R} \cap \class{W}$). The weak equivalences are then the maps $g$ which factor as $g = pi$ where $i$ is a trivial cofibration and $p$ is a trivial fibration.
\end{theorem}

We call such a model structure on $A$, that is, one that is compatible with the exact structure, an \textbf{exact model structure}. Also, we often will denote the model structure by the triple $\class{M} = (\class{Q},\class{W},\class{R})$ and call it a \emph{Hovey triple}.
Recall that a model category is typically now assumed to be bicomplete. But as described in more detail in Section~4 of~\cite{gillespie-exact model structures}, an exact category automatically comes with enough limits and colimits to do the very basics of homotopy theory. For example, if $\cat{A}$ has an exact model structure, then we can construct the left and right homotopy relations without any further assumption on limits/colimits. Moreover, we obtain the Fundamental Theorem of Model Categories asserting that the homotopy category is a localization with respect to the weak equivalences and is equivalent to the full subcategory of cofibrant-fibrant objects, modulo the formal homotopy relation. So we can speak of model structures on exact categories but we won't call it a ``model category'' unless it is bicomplete.

Next, we summarize a main result from~\cite{gillespie-exact model structures} which turns out to be very practical. It is the characterization of the left and right homotopy relations in terms of the cotorsion pairs.

\begin{proposition}\label{prop-left and right homotopic maps in exact model structures}
Assume $\class{A}$ is an exact category with an exact model structure. Let $(\class{Q},\class{R} \cap \class{W})$ and $(\class{Q} \cap \class{W} , \class{R})$ be the corresponding complete cotorsion pairs of Theorem~\ref{them-Hovey's theorem for WIC-exact categories}.
\begin{enumerate}
\item Two maps $f,g \mathcolon X \xrightarrow{} Y$ in $\class{A}$ are right homotopic if and only if $g-f$ factors through a trivially cofibrant object, that is, one in $\class{Q} \cap \class{W}$.
\item Two maps $f,g \mathcolon X \xrightarrow{} Y$ in $\class{A}$ are left homotopic if and only if $g-f$ factors through a trivially fibrant object, that is, one in $\class{R} \cap \class{W}$.
\item Suppose $Y$ is fibrant, that is, $Y \in \class{Q}$. Then two maps $f,g \mathcolon X \xrightarrow{} Y$ in $\class{A}$ are right homotopic if and only if $g-f$ factors through an object of $\class{Q} \cap \class{R} \cap \class{W}$.
\item Suppose $X$ is cofibrant, that is, $X \in \class{Q}$. Then two maps $f,g \mathcolon X \xrightarrow{} Y$ in $\class{A}$ are left homotopic if and only if $g-f$ factors through an object of $\class{Q} \cap \class{R} \cap \class{W}$.
 \item Suppose $X$ is cofibrant and $Y$ is fibrant. Then two maps $f,g \mathcolon X \xrightarrow{} Y$ in $\class{A}$ are homotopic if and only if $g-f$ factors through an object of $\class{Q} \cap \class{R} \cap \class{W}$ if and only if $g-f$ factors through an object of $\class{Q} \cap \class{W}$ if and only if $g-f$ factors through an object of $\class{R} \cap \class{W}$.
\end{enumerate}
\end{proposition}

\subsection{Injective and projective cotorsion pairs}\label{subsec-injective cotorsion pairs}

The notion of projective and injective cotorsion pairs from~\cite{gillespie-recollements} directly carries over to WIC exact categories. An \emph{injective cotorsion pair} in a WIC exact category $\cat{A}$ with enough injectives is, by definition, a complete cotorsion pair $(\class{W},\class{F})$ with $\class{W}$ thick and $\class{W} \cap \class{F}$ equalling the class of injectives in $\cat{A}$. Note that since $\cat{A}$ has enough injectives the cotorsion pair $\class{M} = (\class{W},\class{F})$ is equivalent to an exact model structure $\class{M}$ on $\cat{A}$.

We have the following characterization of injective cotorsion pairs.

\begin{proposition}\label{prop-characterizations of injective cotorsion pairs}
Assume $\class{M} = (\class{W},\class{F})$ is a complete cotorsion pair in a WIC exact category $\cat{A}$ with enough injectives. Then $(\class{W},\class{F})$ is an injective cotorsion pair if and only if $\class{W}$ is thick and contains the injective objects.
\end{proposition}

\begin{proof}
See Section~3 of~\cite{gillespie-recollements}. All the results hold by replacing ``abelian'' with ``WIC exact''.
\end{proof}

Note that the dual notion of a \emph{projective cotorsion pair} relies on the category $\cat{A}$ having enough projectives.

\subsection{The Frobenius category $\boldsymbol{\text{Ch}(\cat{A})_{dw}}$}\label{subsec-chain complexes form a Frobenius cat}

A classic example of a Frobenius category is $\ch$, where $R$ is a ring, along with the degreewise split short exact sequences. The author denotes this exact category by $\ch_{dw}$. It is Frobenius with the projective-injective objects being precisely the contractible complexes and these coincide with the split exact complexes, or equivalently, direct sums of $n$-disks. In fact, $\cha{A}_{dw}$ is always Frobenius no matter what additive category $\cat{A}$ we start with. However, the class of contractible complexes does not, in general, coincide with the class of split exact complexes. In fact the equality of these two classes is equivalent to the additive category $\cat{A}$ being idempotent complete. Recall that $\cat{A}$ is idempotent complete if every idempotent, which is a map $p : A \xrightarrow{} A$ with $p^2 = p$, has a kernel (equivalently, a cokernel). It is easy to see that such a category is weakly idempotent complete, and the converse holds whenever countable coproducts exist in $\cat{A}$~\cite[Remark~7.3]{buhler-exact categories}. If $\cat{A}$ is an exact category with $\cat{A}$ idempotent complete then we will call $\cat{A}$ an idempotent complete exact category, or for short, an IC exact category. Just as with WIC exact categories, it is important to realize that the notion of idempotent completeness is inherent to the underlying category $\cat{A}$, and in particular, does not depend at all on the particular exact structure we are considering on $\cat{A}$. We now briefly summarize with reasons some formal statements making the above more precise.

First, note that the notion of a chain complex certainly makes sense in any additive category. Our convention is that the differential lowers degree, so $\cdots
\xrightarrow{} X_{n+1} \xrightarrow{d_{n+1}} X_{n} \xrightarrow{d_n}
X_{n-1} \xrightarrow{} \cdots$ is a chain complex. We have the following lemma. We leave the proof to the reader, pointing out that some of it appears in~\cite[Lemma~9.1]{buhler-exact categories}.
\begin{lemma}\label{lemma-ch(A)}
Let $\cat{A}$ be an additive category. Note then that $\cha{A}$ is also additive and it is idempotent complete (resp. weakly idempotent complete) if and only if $\cat{A}$ is. If $\cat{A}$ is exact (resp. WIC exact, resp. IC exact, resp. abelian, resp. Grothendieck) then $\cha{A}$ is also exact (resp. WIC exact, IC exact, resp. abelian, resp. Grothendieck) with respect to the short sequences which are exact in each degree.
\end{lemma}
Note too that the notion of chain homotopy also makes sense since $\cat{A}$ is additive. So we have no trouble forming the homotopy category $K(\class{A})$ whose objects are the same as $\cha{A}$ but whose morphisms are homotopy classes of chain maps.
Given $X \in \cha{A}$, the
\emph{suspension of $X$}, denoted $\Sigma X$, is the complex given by
$(\Sigma X)_{n} = X_{n-1}$ and $(d_{\Sigma X})_{n} = -d_{n}$.  The
complex $\Sigma (\Sigma X)$ is denoted $\Sigma^{2} X$ and inductively
we define $\Sigma^{n} X$ for all positive integers. We also set $\Sigma^0 X = X$ and define $\Sigma^{-1}$ by shifting indices in the other direction. Given two chain complexes $X$ and $Y$ we define $\homcomplex(X,Y)$ to
be the complex of abelian groups $ \cdots \xrightarrow{} \prod_{k \in
\Z} \Hom(X_{k},Y_{k+n}) \xrightarrow{\delta_{n}} \prod_{k \in \Z}
\Hom(X_{k},Y_{k+n-1}) \xrightarrow{} \cdots$, where $(\delta_{n}f)_{k}
= d_{k+n}f_{k} - (-1)^n f_{k-1}d_{k}$.
This gives a functor
$\homcomplex(X,-) \mathcolon \cha{A} \xrightarrow{} \textnormal{Ch}(\Z)$. Note that if $\cat{A}$ is an exact category then this functor is left exact,
and it is exact if $X_{n}$ is projective for all $n$. (Recall here that projective means lifting over admissible epimorphisms, so these two claims follow from~\cite[Proposition~11.3]{buhler-exact categories}.) Similarly the contravariant functor $\homcomplex(-,Y)$ sends right
exact sequences to left exact sequences and is exact if $Y_{n}$ is
injective for all $n$. It is an exercise to check that the homology satisfies $H_n[Hom(X,Y)] = K(\class{A})(X,\Sigma^{-n} Y)$.

Now note that any additive $\cat{A}$ is an exact category when we consider it along with the class of all split exact sequences. Thus Lemma~\ref{lemma-ch(A)} tells us that $\cha{A}$, along with the short sequences which are degreewise split, form an exact category which we will denote by $\cha{A}_{dw}$. Following~\cite[Definition~10.1]{buhler-exact categories} a chain complex $X$ over any exact category $\cat{A}$ is called \emph{exact} (or \emph{acyclic}) if the differentials factor as $X_n \twoheadrightarrow Z_{n-1} \rightarrowtail X_{n-1}$ in such a way that each $Z_n \rightarrowtail X_n \twoheadrightarrow Z_{n-1}$ is exact. Note that in particular this implies that the differentials all have a kernel and an image, and that $\ker{d_n} = Z_n = \im{d_{n+1}}$. We will call a chain complex $X$ \emph{split exact} if it is exact in $\cha{A}_{dw}$. It is easy to see that the split exact complexes are characterized as follows. For a given $A \in \cat{A}$, we denote the \emph{$n$-disk on $A$} by $D^n(A)$. This is the complex consisting only of $A \xrightarrow{1_A} A$ concentrated in degrees $n$ and $n-1$. Note that for a given collection $\{A_n\}_{n \in \Z}$, the biproduct $\bigoplus_{n \in \Z} D^n(A_n) = \prod_{n \in \Z} D^n(A_n)$ always exists because it is really just a finite biproduct in each degree. Then a chain complex $X$ is split exact if and only if $X$ is isomorphic to a direct sum of $n$-disks on its cycles. That is,$$X \cong \bigoplus_{n \in \Z} D^n(Z_n) = \prod_{n \in \Z} D^n(Z_n).$$
One can also check that two chain maps are chain homotopic if and only if their difference factors through a contractible complex.

\begin{proposition}\label{prop-Ch(A)_dw is Frobenius}
For any additive category $\cat{A}$, the exact category $\cha{A}_{dw}$ is Frobenius. The projective-injective objects coincide with the contractible complexes and these are precisely the retracts of split exact complexes. Furthermore, the following statements are equivalent:
\begin{enumerate}
\item $\cat{A}$ is idempotent complete.
\item All contractible complexes are split exact.
\end{enumerate}
\end{proposition}

Before the proof we give an example. Let $\cat{A}$ be the additive category of free modules for some ring $R$ for which there exists a projective module $P$ which is not free. Using the Eilenberg swindle~\cite[Corollary~2.7]{lam}, we can construct a free module $F$ for which $P \oplus F \cong F$. Using this construction, we can define a chain complex in $\cha{A}$

$$C = 0 \xrightarrow{} P \oplus F \xrightarrow{} F \xrightarrow{} 0$$ which is contractible, but it is not even exact. This complex $C$ looks like an $n$-disk, and it is projective, but it is not an $n$-disk as the map $P \oplus F \xrightarrow{} F$ is not the identity. The problem of course is that $\class{A}$ is not even weakly idempotent complete.

\begin{proof}
Note that any particular $D^n(A)$ is projective in $\cha{A}_{dw}$ by directly checking that any chain map $D^n(A) \xrightarrow{} Z$ must lift over a degreewise split epimorphism $Y \twoheadrightarrow Z$. Similarly, $D^n(A)$ is injective. Now for a given collection $\{A_n\}_{n \in \Z}$, the direct sum $\bigoplus_{n \in \Z} D^n(A_n)$ exists in the category $\cha{A}$, and since each summand is projective, so is $\bigoplus_{n \in \Z} D^n(A_n)$ by~\cite[Corollary~11.7]{buhler-exact categories}. But for that matter we have that the product $\prod_{n \in \Z} D^n(A_n) = \bigoplus_{n \in \Z} D^n(A_n)$ is injective as well. So indeed $\bigoplus_{n \in \Z} D^n(A_n)$ is projective-injective.

If $X$ is any chain complex, then one can construct an admissible epimorphism $\bigoplus_{n \in \Z} D^n(X_n) \twoheadrightarrow X$. So we have enough projectives. The dual gives enough injectives. Finally, if $P$ is projective in $\cha{A}_{dw}$,  then the admissible epimorphism $\bigoplus_{n \in \Z} D^n(P_n) \twoheadrightarrow P$ splits. In particular, any projective $P$ is contractible, and being a retract of the injective $\bigoplus_{n \in \Z} D^n(P_n)$, we see that $P$ must also be injective by the dual of~\cite[Corollary~11.4]{buhler-exact categories}.
This shows that all projectives are injective and a similar argument will show that all injectives are projective. So $\cha{A}_{dw}$ is Frobenius.

The two statements are equivalent by~\cite[Proposition~10.9]{buhler-exact categories} since the exact complexes in $\cha{A}_{dw}$ coincide with the split exact complexes.

\end{proof}

Of course if $\cat{A}$ is an exact category then we also want to consider model structures on the exact category $\cha{A}$, the exact structure having short exact sequences which are degreewise exact in $\cat{A}$ . Theo B\"uhler has pointed out to the author the following simple way to prove the result in the next corollary.

\begin{corollary}\label{cor-Ch(A) enough projectives}
Let $\cat{A}$ be an exact category. If $\cat{A}$ has enough projectives (resp. injectives) then so does $\cha{A}$. The projective (resp. injective) complexes coincide with the contractible complexes which have projective components and these are precisely the retracts of split exact complexes with projective components. Furthermore, the following statements are equivalent:
\begin{enumerate}
\item $\cat{A}$ is idempotent complete.
\item All projective complexes are split exact with projective components.
\end{enumerate}
\end{corollary}

\begin{proof}
Given any $X$ in $\cha{A}$, we can as in the proof of Proposition~\ref{prop-Ch(A)_dw is Frobenius} find a degreewise split $\bigoplus_{n \in \Z} D^n(X_n) \twoheadrightarrow X$. Of course this must be an admissible epimorphisms in $\cha{A}$ as well. Since $\cat{A}$ has enough projectives we can also find an admissible $P_n \twoheadrightarrow X_n$ for each $n$. $\bigoplus_{n \in \Z} D^n(P_n) \twoheadrightarrow \bigoplus_{n \in \Z} D^n(X_n)$ is an admissible epimorphism. The composition of the two admissible epimorphisms proves the result. The rest of the proof is similar in spirit to that of Proposition~\ref{prop-Ch(A)_dw is Frobenius}.
\end{proof}

Being an exact category, $\cha{A}_{dw}$ comes with a Yoneda Ext functor, which we will denote by $\Ext^1_{dw}$.
The following lemma gives a well-known connection between $\Ext^1_{dw}$ and the hom-complex $\homcomplex$.
\begin{lemma}\label{lemma-homcomplex-basic-lemma}
For chain complexes $X$ and $Y$, we have isomorphisms:
$$\Ext^1_{dw}(X,\Sigma^{(-n-1)}Y) \cong H_n \homcomplex(X,Y) =
K(\cat{A})(X,\Sigma^{-n} Y)$$ In particular, for chain complexes $X$ and $Y$, $\homcomplex(X,Y)$ is
exact iff for any $n \in \Z$, any chain map $f \mathcolon \Sigma^nX \xrightarrow{} Y$ is
homotopic to 0 (or iff any chain map $f \mathcolon X \xrightarrow{} \Sigma^nY$ is homotopic
to 0).
\end{lemma}

In particular, we note that if $\cat{A}$ is an exact category, then Lemma~\ref{lemma-ch(A)} automatically provides us with the two exact structures: $\cha{A}$ and $\cha{A}_{dw}$. For the first, the Yoneda Ext group $\Ext^1_{\cha{A}}(X,Y)$ is the group of (equivalence classes)
of all admissible short exact sequences $Y \rightarrowtail Z \twoheadrightarrow X$ under the Baer
sum. The Yoneda Ext group $\Ext^1_{dw}(X,Y)$ is the subgroup of $\Ext^1_{\cha{A}}(X,Y)$ consisting precisely of those admissible short exact sequences which are split in each
degree.

%%%%%%%%%%%%%%%%%%%%%%%%%%%%%%%%%%%%%%%%%%%%%%%%%%%%%%%%%%%%%%%%%%%%%%%%%
%%%%%%%%%%%%%%%%%%%%%%%%%%%%%%%%%%%%%%%%%%%%%%%%%%%%%%%%%%%%%%%%%%%%%%%%
%%%%%%%%%%%%%%%%%%%%%%%%%%%%%%%%%%%%%%%%%%%%%%%%%%%%%%%%%%%%%%%%%%%%%%%%%

We will say that a class of chain complexes $\class{X}$ is \emph{closed under suspensions} if for any $X \in \class{X}$ we have $\Sigma^n X \in \class{X}$ for any integer $n$. We say that $\class{X}$ is \emph{closed under positive suspensions} if $\Sigma^n X \in \class{X}$ whenever $n \geq 0$, and \emph{closed under negative suspensions} if $\Sigma^n X \in \class{X}$ whenever $n \leq 0$. Recall that in any Frobenius category $\cat{A}$, the formal suspension $\Sigma A$ of an object $A$ is defined to be an object fitting into a short exact sequence $A \rightarrowtail W \twoheadrightarrow \Sigma A$ where $W$ is injective. This cosyzygy, $\Sigma A$, is unique up to a unique isomorphism in the stable category, and the dual notion of the loop functor $\Omega A$ serves as the inverse. Clearly, for a chain complex $X$, the complex $\Sigma X$ coincides with this formal notion of suspension while $\Sigma^{-1} X$ serves as the formal loop. This is seen inside the proof of then next lemma.

\begin{lemma}\label{lemma-suspensions are cosyzygies}
Let $\class{X}$ be a class of chain complexes. Assume $\class{X}$ is closed under direct sums and direct summands and contains the contractible complexes. Then
\begin{enumerate}
\item $\class{X}$ is closed under positive suspensions if and only if $\class{X}$ is cosyzygy closed in the exact category $\cha{A}_{dw}$.
\item $\class{X}$ is closed under negative suspensions if and only if $\class{X}$ is syzygy closed in the exact category $\cha{A}_{dw}$.
\end{enumerate}
\end{lemma}

\begin{proof}
We prove (1). First say $\class{X}$ is cosyzygy closed in $\cha{A}_{dw}$. Then for any $X \in \class{X}$ we have an obvious degreewise split short exact sequence $X \rightarrowtail \oplus_{n \in \Z} D^{n+1}(X_n) \twoheadrightarrow \Sigma X$. Since $\oplus_{n \in \Z} D^{n+1}(X_n)$ is contractible it is in $\class{X}$ and by hypothesis we get the suspension $\Sigma X \in \class{X}$ too. On the other hand, let $\class{X}$ be closed under positive suspensions and let $X \in \class{X}$. Say we have an exact $X \rightarrowtail W \twoheadrightarrow Z$ where $W$ is injective in $\cha{A}_{dw}$.  We wish to show $Z \in \class{X}$. But as above we also have the obvious degreewise split short exact sequence $X \rightarrowtail \oplus_{n \in \Z} D^{n+1}(X_n) \twoheadrightarrow \Sigma X$. So by Schanuel's Lemma~\ref{lemma-schanuel's lemma} we get an isomorphism of complexes $\oplus_{n \in \Z} D^{n+1}(X_n) \oplus Z \cong W \oplus \Sigma X$. The hypotheses imply $W \oplus \Sigma X \in \class{X}$ and so $\oplus_{n \in \Z} D^{n+1}(X_n) \oplus Z \in \class{X}$, and therefore $Z \in \class{X}$.

\end{proof}

\begin{lemma}[Schanuel's Lemma]\label{lemma-schanuel's lemma}
Let $\cat{A}$ be an exact category with enough injectives and let $A \in
 \cat{A}$. Given two exact sequences $A \rightarrowtail I_1 \twoheadrightarrow Z_1 $ and $A \rightarrowtail I_2 \twoheadrightarrow Z_2 $ where $I_1$ and $I_2$ are injective we have $I_1 \oplus Z_2 \cong I_2 \oplus Z_1$. The dual statement holds whenever $\cat{A}$ has enough projectives.

\end{lemma}

\begin{proof}
Take the pushout diagram shown below.
$$\begin{tikzcd}
A \arrow[tail]{r} \arrow[tail]{d} & I_1 \arrow[two heads]{r} \arrow[tail]{d} & Z_1 \arrow[equal]{d} \\
I_2 \arrow[tail]{r} \arrow[two heads]{d} & P \arrow[two heads]{r} \arrow[two heads]{d} & Z_1  \\
Z_2 \arrow[equal]{r} & Z_2
\end{tikzcd}$$ Since $I_1$ and $I_2$ are injective we get both $P \cong I_1 \oplus Z_2$ and $P \cong I_2 \oplus Z_1$. Weak idempotent completeness is not needed to get the direct sum; the sequences are all kernel-cokernel pairs.
\end{proof}

\section{Recollements from exact model structures}\label{sec-recollements from exact models strucs}

The starting point of this paper is a more general and practical version of Theorem~4.6 from~\cite{gillespie-recollements} which we prove in this section. It produces a recollement situation from three injective cotorsion pairs.  The first key to this version is the realization that all the results of Sections~3 and~4 of~\cite{gillespie-recollements} hold, with the same proofs, if one replaces the word ``abelian'' with the more general notion of ``WIC exact'' as described in Section~\ref{sec-preliminaries}. So in particular, we have in this setting Becker's right and left localization constructions from~\cite{becker}, which correspond to left and right Bousfield localization. Briefly, given two injective cotorsion pairs $\class{M}_1 = (\class{W}_1,\class{F}_1)$ and $\class{M}_2 = (\class{W}_2,\class{F}_2)$ with $\class{F}_2 \subseteq \class{F}_1$, Becker defined their \emph{right localization}, denoted $\class{M}_1/\class{M}_2$, to be the (exact) model structure given by the Hovey triple $\class{M}_1/\class{M}_2 = (\class{W}_2,\class{W},\class{F}_1)$. That is, he explicitly describes a thick class $\class{W}$ making $(\class{W}_2 \cap \class{W} , \class{F}_1)$ and $(\class{W}_2, \class{W} \cap \class{F}_1)$ each complete cotorsion pairs. Moreover, he shows that this is the right Bousfield localization of $\class{M}_1$ by $\class{M}_2$. The second key feature to Theorem~\ref{them-recollements in krause form} is a cosmetic change from~\cite[Theorem~4.6]{gillespie-recollements} which results by focusing on the homotopy category being a triangulated localization. We describe this aspect more now.

\subsection{The triangulated localization $\cat{A}/\class{W}$}\label{subsec-triangulated localization}

Let $\cat{A}$ be a WIC exact category. We wish to see that the homotopy category of an exact model structure on $\cat{A}$ satisfies a universal property saying that it is the triangulated localization of $\cat{A}$ with respect to the class $\class{W}$ of trivial objects. This is fundamental but doesn't seem to be approached directly in the literature. Since it is easiest to explain in the context of the special \emph{injective} model structures, and since this is all we will need for this paper (as well as the projective duals), we will only focus on this special case. First, a lemma.

\begin{lemma}\label{lemma-characterization of weak equivalences in exact categories}
Say $\cat{A}$ is a WIC exact category with an exact model structure. Let $\class{W}$ denote the class of trivial objects. Then a map $f$ is a weak equivalence if and only if it factors as an admissible monomorphism with cokernel in $\class{W}$ followed by an admissible epimorphism with kernel in $\class{W}$.
\end{lemma}

\begin{proof}
Hovey proved in~\cite[Lemma~5.8]{hovey} that in an abelian model category, a monomorphism is a weak equivalence iff its cokernel is trivial. The author checked that the analogous Lemma holds in the language of exact categories when writing~\cite{gillespie-exact model structures}. So the current lemma follows from this and an application of the two out of three axiom.

\end{proof}

Note that the Lemma above says that if $\cat{A}$ has an exact model structure given by the Hovey triple $\class{M} = (\class{Q},\class{W},\class{R})$, then the weak equivalences are completely determined by the class $\class{W}$ of trivial objects. It therefore makes sense to denote the homotopy category $\Ho{\class{M}}$, which is the localization with respect to the weak equivalences, by $\cat{A}/\class{W}$. We will do this. Why not use $\Ho{\class{M}}$? The answer is that we will use both because a recollement situation is the result of having 5 model structures on $\cat{A}$ at the same time, some having the same class $\class{W}$ as the trivial objects. Their localizations $\cat{A}/\class{W}$ are the same. It is their different cofibrant-fibrant objects that allow for different representations of this localization. It is therefore much more convenient to let $\cat{A}/\class{W}$ denote the localization, while $\Ho{\class{M}}$ denotes the localization with all the added structure that comes with having a model structure. It is this perspective which is the key to proving the recollement diagram in Theorem~\ref{them-recollements in krause form} and this diagram is more in line with the applications appearing in the literature.

So to emphasize, $\gamma : \cat{A} \xrightarrow{} \cat{A}/\class{W}$, will denote the canonical functor to the localization category $\cat{A}/\class{W}$, where we have formally inverted the weak equivalences. When confusion can arise from more than one possible class of trivial objects, we will denote $\gamma$ by $\gamma_{\class{W}}$. The existence of a Hovey triple $\class{M} = (\class{Q},\class{W},\class{R})$ implies that $\cat{A}/\class{W}$ exists as an actual category (with small Hom sets). The fundamental theorem states $$\cat{A}/\class{W} = \Ho{\class{M}} \cong (\class{Q} \cap \class{R})/\sim \,$$ where the $\sim$ is characterized as in part~(5) of Proposition~\ref{prop-left and right homotopic maps in exact model structures}.

Of course having a Hovey triple $\class{M} = (\class{Q},\class{W},\class{R})$ means that $\cat{A}/\class{W}$ has much more structure that just being a category. The homotopy category of any category with a zero object must be a triangulated category as shown in~\cite[Chapter~7]{hovey-model-categories}. Here are some basic facts on the triangulated structure that the author has taken from~\cite[p.~586]{hovey}. First, the suspension $\Sigma X$, of a cofibrant object $X \in \class{Q}$, is computed by taking the cokernel of a cofibration $X \hookrightarrow CX$ where $CX$ is trivial. (A different choice of $CX$ would result in a weak equivalence between the two possible $\Sigma X$'s.) Second, suspensions can be computed in \emph{any} model structure since they are preserved by Quillen equivalences. As a consequence, if we have more that one model structure on $\cat{A}$ with trivial objects $\class{W}$, then $\cat{A}/\class{W}$ has the same triangulated structure regardless of the model structure we are using.

Since all of our localizations $\cat{A}/\class{W}$ can be realized from an injective (or projective) cotorsion pair, we summarize some important information in the following Proposition.

\begin{proposition}\label{prop-triangulated localization}
Let $\class{M} = (\class{W},\class{F})$ be an injective cotorsion pair in a WIC exact category $\cat{A}$ with enough injectives.
\begin{enumerate}
\item $\class{F}$ naturally inherits the structure of a Frobenius category with the projective-injective objects being precisely the injectives from $\cat{A}$.

\item The functor $\gamma : \cat{A} \xrightarrow{\gamma_{\class{W}}} \cat{A}/\class{W} = \Ho{\class{M}} \cong \class{F}/\sim$ is exact in the sense that it takes short exact sequences in $\cat{A}$ to exact triangles in $\Ho{\cat{M}}$.

\item $\gamma$ is universal among triangulated categories $\class{T}$ which ``kill'' $\class{W}$. That is, given another exact $F : \cat{A} \xrightarrow{} \class{T}$ with $F(\class{W})=0$, it factors through $\gamma$.
\end{enumerate}

\end{proposition}

\begin{proof}
For (1), recall that a Frobenius category is an exact category with enough projectives and injectives and in which the projective and injective objects coincide. Since in our case, $\cat{A}$ is weakly idempotent complete, retracts and direct summands coincide. So the class $\class{F}$ is closed under direct summands. It is therefore easy to see that $\class{F}$ naturally inherits a WIC exact structure where the admissible short exact sequences are the ones from $\cat{A}$ but with all three terms in $\class{F}$. (This is proved in~\cite[Lemma~5.1]{gillespie-exact model structures}.) For any injective cotorsion pair $(\class{W},\class{F})$, the class $\class{F}$ contains the injectives and is coresolving. So it is clear that $\class{F}$ has enough injectives, with the injectives being those in $\cat{A}$. Since $(\class{W},\class{F})$ is an injective cotorsion pair we have that the injectives lie in $\class{W}$, and so must be projective with respect to the admissible short exact sequences in $\class{F}$. Next, let $F \in \class{F}$ be arbitrary. Using enough projectives from $(\class{W},\class{F})$, find an admissible short exact sequence $F' \rightarrowtail W \twoheadrightarrow F$ with $W \in \class{W}$ and $F' \in \class{F}$. Then $W \in \class{W} \cap \class{F}$ must be injective. So $\class{F}$ also has enough projectives. If $F$ in the above short exact sequence happened to be a projective object in $\class{F}$, then the sequence would split. This allows us to conclude that the projectives in $\class{F}$ coincide with the injectives. So $\class{F}$ is a Frobenius category.

(2) follows from~\cite[Lemma~1.4.4]{becker} which says that given a short exact sequence $A \rightarrowtail B \twoheadrightarrow C$ in $\cat{A}$ there exists a commutative diagram in $\cat{A}$ of short exact sequences
$$\begin{tikzcd}
              A \arrow[tail]{r}    \arrow[tail]{d}      & B \arrow[two heads]{r} \arrow[tail]{d} & C \arrow[tail]{d} \\
F \arrow[tail]{r} \arrow[two heads]{d} & F' \arrow[two heads]{r} \arrow[two heads]{d} & F'' \arrow[two heads]{d} \\
W \arrow[tail]{r}  & W' \arrow[two heads]{r} & W''
\end{tikzcd}$$ with $F,F',F'' \in \class{F}$ and $W,W',W'' \in \class{W}$.
In more detail, denote the maps in the sequences by $A \xrightarrow{f} B \xrightarrow{g} C$ and $F \xrightarrow{k} F' \xrightarrow{h} F''$. Then since $\class{F}$ is Frobenius, there is a map $s : F'' \xrightarrow{} \Sigma F$ in $\cat{A}$ such that $F \xrightarrow{[k]} F' \xrightarrow{[h]} F'' \xrightarrow{[s]} \Sigma F$ is an exact triangle in the stable category $\class{F}/\sim \,$. Our goal is to find a map $C \xrightarrow{} \Sigma A$ in $\Ho{\class{A}}$ and to show that the triangle $A \xrightarrow{\gamma(f)} B \xrightarrow{\gamma(g)} C \xrightarrow{} \Sigma A$ is isomorphic to this exact triangle, showing that the latter triangle is also exact.

Consider again the short exact sequence $A \rightarrowtail F \twoheadrightarrow W$ which is the left vertical column of the commutative diagram. Use enough injectives to write
$F \rightarrowtail I \twoheadrightarrow \Sigma F$ with $I$ injective, making $\Sigma F$ the suspension. Take the pushout of $I \leftarrow F \rightarrow W$ to get a commutative diagram where the upper right square is bicartesian (push-pull square). The rows and columns are admissible short exact sequences by~\cite[Proposition~2.12]{buhler-exact categories}.
$$\begin{tikzcd}
A \arrow[tail]{r}    \arrow[equal]{d}      & F \arrow[two heads]{r} \arrow[tail]{d} & W \arrow[tail]{d} \\
A \arrow[tail]{r}  & I \arrow[two heads]{r} \arrow[two heads]{d} & \Sigma A \arrow[two heads]{d} \\
          & \Sigma F \arrow[equal]{r} & \Sigma F
\end{tikzcd}$$
Now, by definition, $\Sigma A$ is just the pushout. But we denote it $\Sigma A$ because by the remarks before the statement of this proposition, $\Sigma A$ \emph{is} the suspension of $A$ (since $A$ is cofibrant and $I$ is trivial). Denoting the map $\Sigma A \xrightarrow{p} \Sigma F$, we see that $p$ is an admissible epimorphism with trivial kernel, so it is a weak equivalence. This gives us a map of triangles in the homotopy category
$$\begin{CD}
A       @>\gamma(f)>>     B      @>\gamma(g)>>    C  @>>>     \Sigma A  \\
@VVV             @VVV       @VVV             @VV\gamma(p)V         \\
F      @>[k]>>  F'  @>[h]>>  F'' @>[s]>>  \Sigma F \\
@VVV            @VVV      @VVV \\
W      @.      W'      @.    W'' \\
\end{CD}$$
Since $\ker{p} = W$ and all of $W,W',W''$ are in $\class{W}$, this is an isomorphism of triangles in $\Ho{\class{A}}$. So the map $C \xrightarrow{} \Sigma A$ that we have sought is just defined so that the right square commutes in $\Ho{\class{A}}$ (using that $\gamma(p)$ is an isomorphism here).

For (3), say $F : \cat{A} \xrightarrow{} \class{T}$ is given with $\class{T}$ triangulated and $F(\class{W}) = 0$. Then given a trivial cofibration $A \hookrightarrow B \twoheadrightarrow W$ we have an exact triangle $FA \xrightarrow{} FB \xrightarrow{} FW  \xrightarrow{} \Sigma FA$. Since $FW = 0$ and $\class{T}$ is triangulated, it follows that $FA \xrightarrow{} FB$ must be an isomorphism. So $F$ sends trivial cofibrations to isomorphisms. Similarly, it sends trivial fibrations to isomorphisms. Therefore \emph{any} weak equivalence, which must factor as a trivial cofibration followed by a trivial fibration, must also be sent to an isomorphism. We thus have proved that $F$ sends all weak equivalences to isomorphisms, and so by the fundamental fact that $\Ho{\cat{A}}$ is the localization with respect to the weak equivalences, we get that $F$ factors uniquely through $\gamma_{\class{W}}$.

\end{proof}

\subsection{Recollements from cotorsion pairs}

The following functor is crucial to the recollement diagram appearing in Theorem~\ref{them-recollements in krause form}.

\begin{lemma}\label{lemma-quotient map}
Let $\cat{A}$ be a WIC exact category with enough injectives and suppose we have injective cotorsion pairs $\class{M} = (\class{W}, \class{F})$ and $\class{M}' = (\class{W}', \class{F}')$ with $\class{F}' \subseteq \class{F}$. Then the \textbf{\emph{quotient functor}} $Q : \class{F}/\sim \, \xrightarrow{} \cat{A}/\class{W}'$ defined by $Q([f]) = \gamma_{\class{W}'}(f)$ is well defined.

\end{lemma}

\begin{proof}
Say $f \sim g$ are maps in $\class{F}/\sim$. Then by Proposition~\ref{prop-left and right homotopic maps in exact model structures} we see that $g-f$ factors through an injective object $I$. But since $\class{M}' = (\class{W}', \class{F}')$ is an injective cotorsion pair we have $I \in \class{W}'$ from Proposition~\ref{prop-characterizations of injective cotorsion pairs}. Thus the functor $\gamma  = \gamma_{\class{W}'} : \cat{A} \xrightarrow{} \cat{A}/\class{W}'$ from Proposition~\ref{prop-triangulated localization} satisfies $\gamma(g-f) = \gamma(g)-\gamma(f)$ factors through 0 in $\cat{A}/\class{W}'$.

\end{proof}

Finally, the language of special precovers and preenvelope from~\cite{enochs-jenda-book} will be useful to describe derived functors. For example, when coming across an injective cotorsion pair $\class{M}_1 = (\class{W}_1, \class{F}_1)$, the notation E$(\class{M}_1)$ means to take a special $\class{F}_1$-preenvelope by using enough injectives of the cotorsion pair $\class{M}_1 = (\class{W}_1, \class{F}_1)$. This corresponds to a fibrant replacement in the corresponding model structure on $\cat{A}$, as we will recall in the proof of Theorem~\ref{them-recollements in krause form}. Similarly, for another injective pair $\class{M}_2 = (\class{W}_2, \class{F}_2)$ with $\class{F}_2 \subseteq \class{F}_1$, the notation C$(\class{M}_2)$ means to take a special $\class{W}_2$-precover. This corresponds to cofibrant replacement in $\class{M}_1/\class{M}_2 = (\class{W}_2, \class{W}, \class{F}_1)$.

\begin{theorem}[Injective Recollement Theorem]\label{them-recollements in krause form}
Let $\cat{A}$ be a WIC exact category with enough injectives and suppose we have three injective cotorsion pairs $$\class{M}_1 = (\class{W}_1, \class{F}_1) , \ \ \ \class{M}_2 = (\class{W}_2, \class{F}_2) , \ \ \ \class{M}_3 = (\class{W}_3, \class{F}_3)$$ such that $\class{F}_2 , \class{F}_3 \subseteq  \class{F}_1$. If $\class{W}_3 \cap \class{F}_1 = \class{F}_2$ (or
equivalently, $\class{W}_2 \cap \class{W}_3 = \class{W}_1$ and $\class{F}_2 \subseteq \class{W}_3$), then $\class{M}_1/\class{M}_2$ is Quillen equivalent to $\class{M}_3$ and $\class{M}_1/\class{M}_3$ is Quillen equivalent to $\class{M}_2$. In fact, we have a recollement as shown below.
\[
\begin{tikzpicture}[node distance=3.5cm]
\node (A) {$\mathcal{F}_2/\sim$};
\node (B) [right of=A] {$\mathcal{F}_1/\sim$};
\node (C) [right of=B] {$\class{A}/\class{W}_3$};
\draw[<-,bend left=40] (A.20) to node[above]{\small E$(\class{M}_2)$} (B.160);
\draw[->] (A) to node[above]{\small $I$} (B);
\draw[<-,bend right=40] (A.340) to node [below]{\small C$(\class{M}_3)$} (B.200);
\draw[<-,bend left] (B.20) to node[above]{\small $\lambda = \text{C}(\class{M}_2) \circ \text{E}(\class{M}_1)$} (C.160);
\draw[->] (B) to node[above]{\small Q} (C);
\draw[<-,bend right] (B.340) to node [below]{\small $\rho = \text{E}(\class{M}_3)$} (C.200);
\end{tikzpicture}
\]
Here, the functor $I$ is just inclusion while $Q$ is the quotient functor of Lemma~\ref{lemma-quotient map}. We point out that $\lambda$ has essential image $(\class{W}_2 \cap \class{F}_1)/\sim$ while  $\rho$ has essential image $\class{F}_3/\sim$ and they provide an equivalence $\lambda : \class{F}_3/\sim \, \longleftrightarrow  (\class{W}_2 \cap \class{F}_1)/\sim \, : \rho$.
\end{theorem}

\begin{proof}
We start by doing the same thing as in the proof of~\cite[Theorem~4.6]{gillespie-recollements}. That is, we apply~\cite[Corollary~1.4.5]{becker} to the three injective cotorsion pairs and are led to the diagram.
\[
\begin{tikzpicture}[node distance=3.5 cm, auto]
\node (A)  {$\Ho{\class{M}_2}$};
\node (D) [below of=A] {$\Ho{\class{M}_1/\class{M}_3}$};
\node (B) [right of=A] {$\Ho{\class{M}_1}$};
\node (C) [right of=B] {$\Ho{\class{M}_1/\class{M}_2}$};
\node (E) [right of=D] {$\Ho{\class{M}_1}$};
\node (F) [right of=E] {$\Ho{\class{M}_3}$};

%
% Horizontal arrows
%
\draw[<-] (A.10) to node {L\,id} (B.170);
\draw[->] (A.350) to node [swap] {R\,id} (B.190);
\draw[->] (D.6) to node {L\,id} (E.170);
\draw[<-] (D.353) to node [swap] {R\,id} (E.190);
\draw[->] (E.10) to node {L\,id} (F.170);
\draw[<-] (E.350) to node [swap] {R\,id} (F.190);
\draw[<-] (B.10) to node {L\,id} (C.173);
\draw[->] (B.350) to node [swap] {R\,id} (C.187);
%
% Vertical Arrows
%
\draw[-] (B.280) to node {} (E.80);
\draw[-] (B.260) to node [swap] {} (E.100);
\draw[->] (A.290) to node {\textnormal{R\,id}} (D.70);
\draw[<-] (A.250) to node [swap] {\textnormal{L\,id}} (D.110);
\draw[<-] (C.290) to node {R\,id} (F.70);
\draw[->] (C.250) to node [swap] {L\,id} (F.110);
\end{tikzpicture}
\]
These are all derived adjunctions coming from the of the identity Quillen adjunctions. The cited corollary of Becker tells us that since $\class{F}_2 \subseteq \class{F}_1$, the top row is a colocalization sequence, and since $\class{F}_3 \subseteq \class{F}_1$, the bottom row a localization sequence. As explained in~\cite[Theorem~4.6]{gillespie-recollements}, the hypothesis $\class{W}_3 \cap \class{F}_1 = \class{F}_2$ is equivalent to the hypothesis $\class{W}_2 \cap \class{W}_3 = \class{W}_1$ and $\class{F}_2 \subseteq \class{W}_3$, and easily leads to the fact that the vertical functors are all equivalences. In particular, $\Ho{\class{M}_3} \cong \Ho{\class{M}_1/\class{M}_2}$ and $\Ho{\class{M}_2} \cong \Ho{\class{M}_1/\class{M}_3}$. When we analyze what these functors are doing on the level of the cofibrant-fibrant subcategories we are led to the following diagram.
\[
\begin{tikzpicture}[node distance=3.5 cm, auto]
\node (A)  {$\class{F}_2/\sim$};
\node (D) [below of=A] {$\class{F}_2/\sim$};
\node (B) [right of=A] {$\class{F}_1/\sim$};
\node (C) [right of=B] {$(\class{W}_2 \cap \class{F}_1)/\sim$};
\node (E) [right of=D] {$\class{F}_1/\sim$};
\node (F) [right of=E] {$\class{F}_3/\sim$};

%
% Horizontal arrows
%
\draw[<-] (A.17) to node {E$(\class{M}_2)$} (B.163);
\draw[->] (A.343) to node [swap] {Inclusion} (B.197);
\draw[->] (D.17) to node {Inclusion} (E.163);
\draw[<-] (D.343) to node [swap] {C$(\class{M}_3)$} (E.197);
\draw[->] (E.17) to node {E$(\class{M}_3)$} (F.163);
\draw[<-] (E.343) to node [swap] {Inclusion} (F.197);
\draw[<-] (B.17) to node {Inclusion} (C.171);
\draw[->] (B.343) to node [swap] {C$(\class{M}_2)$} (C.189);
%
% Vertical Arrows
%
\draw[-] (B.280) to node {} (E.80);
\draw[-] (B.260) to node [swap] {} (E.100);
\draw[->] (A.290) to node {\textnormal{id}} (D.70);
\draw[<-] (A.250) to node [swap] {\textnormal{id}} (D.110);
\draw[<-] (C.290) to node {C$(\class{M}_2)$} (F.70);
\draw[->] (C.250) to node [swap] {E$(\class{M}_3)$} (F.110);
\end{tikzpicture}
\]
Let $Q : \class{F}_1/\sim \, \xrightarrow{} \cat{A}/\class{W}_3$ be the quotient map of Lemma~\ref{lemma-quotient map}. The key to proving the stated version of the theorem is to realized that $Q$ factors through the vertical arrows to the far right, as we will now show.

We first note that the hypothesis implies $\class{M}_1/\class{M}_2 = (\class{W}_2, \class{W}_3, \class{F}_1)$ and so $\cat{A}/\class{W}_3 = \Ho{\class{M}_1/\class{M}_2} \cong (\class{W}_2 \cap \class{F}_1)/\sim \,$. Now consider the diagram below where the vertical maps simply reflect the canonical equivalence $\Ho{\class{M}_1/\class{M}_2} \cong (\class{W}_2 \cap \class{F}_1)/\sim$ since $\textnormal{C}(\class{M}_2) \circ \textnormal{E}(\class{M}_1)$ represents fibrant replacement followed by cofibrant replacement in the model structure $\class{M}_1/\class{M}_2$.
\[
\begin{tikzpicture}[node distance=3.5 cm, auto]
\node (B) [right of=A] {$\class{F}_1/\sim$};
\node (C) [right of=B] {$(\class{W}_2 \cap \class{F}_1)/\sim$};
\node (E) [right of=D] {$\class{F}_1/\sim$};
\node (F) [right of=E] {$\class{A}/\class{W}_3$};

%
% Horizontal arrows
%

%
\draw[->] (E.0) to node {$Q$} (F.180);
\draw[<-] (B.17) to node {Inc} (C.171);
\draw[->] (B.343) to node [swap] {C$(\class{M}_2)$} (C.189);
%
% Vertical Arrows
%
\draw[-] (B.280) to node {} (E.80);
\draw[-] (B.260) to node [swap] {} (E.100);
\draw[<-] (C.300) to node {$\textnormal{C}(\class{M}_2) \circ \textnormal{E}(\class{M}_1)$} (F.60);
\draw[->] (C.240) to node [swap] {Inc} (F.120);
\end{tikzpicture}
\]

We digress from the proof to review how the cofibrant replacement functors such as $\class{F}_1/\sim \xrightarrow{\textnormal{C}(\class{M}_2)} (\class{W}_2 \cap \class{F}_1)/\sim$ actually work. So say $[f] : A \xrightarrow{} B$ is any morphism in $\class{F}_1/\sim\,$. Using enough projectives of the cotorsion pair $\class{M}_2 = (\class{W}_2,\class{F}_2)$ we can take short exact sequences $F_2  \rightarrowtail W_2 \twoheadrightarrow  A$ and $F'_2  \rightarrowtail W'_2 \twoheadrightarrow  B$ with $W_2, W'_2 \in \class{W}_2$ and $F_2, F'_2 \in \class{F}_2$. The fact that $(\class{W}_2,\class{F}_2)$ is a cotorsion pair implies that there is always a lift $\tilde{f} : W_2 \xrightarrow{} W'_2$ as indicated in the diagram below.
$$\begin{CD}
F_2       @>>>     F'_2       \\
@VVV             @VVV               \\
W_2       @>\tilde{f}>>  W'_2  \\
@V p_A VV            @VV p_B V       \\
A       @>f>>       B    \\
\end{CD}$$ Given such an $f$, the map $\tilde{f}$ with this property is unique in the
stable category $(\class{W}_2 \cap \class{F}_1)/\sim$. To see this, suppose $g$ is another map making the square commute. We wish to show $\tilde{f} \sim g$. But since both $\tilde{f}$ and $g$ each make the square commute we get $p_{B}(\tilde{f}-g) = 0$. So $\tilde{f}-g$ factors through the kernel $F'_2$. But since $(\class{W}_2,\class{F}_2)$ is an injective cotorsion pair we can find a s.e.s $F''_2 \rightarrowtail I \twoheadrightarrow F'_2$ with $I$ injective and $F''_2 \in \class{F}_2$. It is easy to see now that $\tilde{f}-g$ doesn't just factor through $F'_2$ but factors through the injective $I$. This proves $\tilde{f} \sim g$. Similarly, recall that this association depends only on the homotopy class of $f$, confirming that $[f] \mapsto [\tilde{f}]$ is well-defined. To see this, say $f_1 - f_2$ factors as $st$ through an injective $I_1$. Since $I_1$ is trivial, we have $\Ext^1(I,F'_2) = 0$ and so $s$ lifts over $p_{B}$, meaning we have a $I_1 \xrightarrow{w} W'_2$ with $p_{B}w = s$. This leads us to a map $\phi : A \xrightarrow{} W'_2$ with $p_{B}\phi = f_1 - f_2$. But then we turn around and note that $(\tilde{f}_1-\tilde{f}_2) - \phi p$ must factor through the kernel $F'_2$ of $p_{B}$. It follows then for the same reasons as above that $\tilde{f}_1-\tilde{f}_2$ doesn't just factor through $F'_2$ but that it must factor as $ab$ through another injective $I_2$. But since $\tilde{f}_1-\tilde{f}_2 = \phi p + ab$ where $\phi p$ factors through $I_1$ and $ab$ factors through $I_2$ we get that $\tilde{f}_1-\tilde{f}_2$ factors through the injective $I_1 \oplus I_2$. Putting all these properties together allows for a well-defined functor $\textnormal{C}(\class{M}_2)([f]) = [\tilde{f}]$.

Now going back to the proof, we see that in the above paragraph, the objects $F_2,F'_2 \in \class{F}_2 \subseteq \class{W}_3$ are trivial in  $\class{M}_1/\class{M}_2 = (\class{W}_2, \class{W}_3, \class{F}_1)$. We immediately conclude that we have a natural isomorphism $$\{\,p_A\,\} : \textnormal{Inc} \circ \textnormal{C}(\class{M}_2) \cong Q.$$
One can now check that we have proved we have the colocalization sequence:
\[
\begin{tikzpicture}[node distance=3.5 cm, auto]
\node (A)  {$\class{F}_2/\sim$};
\node (B) [right of=A] {$\class{F}_1/\sim$};
\node (C) [right of=B] {$\cat{A}/\class{W}_3$};
\

%
% Horizontal arrows
%
\draw[<-] (A.17) to node {E$(\class{M}_2)$} (B.163);
\draw[->] (A.343) to node [swap] {Inclusion} (B.197);

\draw[<-] (B.17) to node {$\textnormal{C}(\class{M}_2) \circ \textnormal{E}(\class{M}_1)$} (C.163);
\draw[->] (B.343) to node [swap] {$Q$} (C.197);

\end{tikzpicture}
\]
Similarly, we consider the diagram below where the vertical maps simply reflect the canonical equivalence $\Ho{\class{M}_3} \cong \class{F}_3/\sim$ since $\textnormal{E}(\class{M}_3)$ represents fibrant replacement in the model structure $\class{M}_3$.
\[
\begin{tikzpicture}[node distance=3.5 cm, auto]
\node (B) [right of=A] {$\class{F}_1/\sim$};
\node (C) [right of=B] {$\class{A}/\class{W}_3$};
\node (E) [right of=D] {$\class{F}_1/\sim$};
\node (F) [right of=E] {$\class{F}_3/\sim$};

%
% Horizontal arrows
%

%
\draw[->] (B.0) to node {$Q$} (C.180);
\draw[->] (E.17) to node {$\textnormal{E}(\class{M}_3)$} (F.163);
\draw[<-] (E.343) to node [swap] {Inc} (F.197);
%
% Vertical Arrows
%
\draw[-] (B.280) to node {} (E.80);
\draw[-] (B.260) to node [swap] {} (E.100);
\draw[->] (C.300) to node {$\textnormal{E}(\class{M}_3)$} (F.60);
\draw[<-] (C.240) to node [swap] {Inc} (F.120);
\end{tikzpicture}
\]
As above we see that there is a natural isomorphism $\{\,j_A\,\} : Q \cong \textnormal{Inc} \circ \textnormal{E}(\class{M}_3)$
due to the commutative diagram below.
$$\begin{CD}
A       @>j_A>>  F_3 @>>>     W_3 \\
@V f VV            @VV \hat{f} V       \\
B       @>j_B>>       F'_3  @>>>     W'_3  \\
\end{CD}$$
This shows that we have a localization sequence as below and proves the theorem.
\[
\begin{tikzpicture}[node distance=3.5 cm, auto]
\node (A)  {$\class{F}_2/\sim$};
\node (B) [right of=A] {$\class{F}_1/\sim$};
\node (C) [right of=B] {$\cat{A}/\class{W}_3$};
\

%
% Horizontal arrows
%
\draw[<-] (A.17) to node {Inclusion} (B.163);
\draw[->] (A.343) to node [swap] {C$(\class{M}_3)$} (B.197);

\draw[<-] (B.17) to node {$Q$} (C.163);
\draw[->] (B.343) to node [swap] {$\textnormal{E}(\class{M}_3)$} (C.197);

\end{tikzpicture}
\]

\end{proof}

There is a projective dual to Theorem~\ref{them-recollements in krause form} as well. We state it now for easy reference later. Here we recall that given two projective cotorsion pairs $\class{M}_1 = (\class{C}_1,\class{W}_1)$ and $\class{M}_2 = (\class{C}_2,\class{W}_2)$ with $\class{C}_2 \subseteq \class{C}_1$, Becker defined in~\cite{becker} their \emph{left localization}, denoted $\class{M}_2\backslash\class{M}_1$ (note the notation, suggesting ``left''), to be a particular Hovey triple $\class{M}_2\backslash\class{M}_1 = (\class{C}_1,\class{W},\class{W}_2)$. He showed $\class{M}_2\backslash\class{M}_1$ is the left Bousfield localization of $\class{M}_1$ by $\class{M}_2$.

\begin{theorem}[Projective Recollement Theorem]\label{them-recollements theorem projective version}
Let $\cat{A}$ be a WIC exact category with enough projectives and suppose we have three projective cotorsion pairs $$\class{M}_1 = (\class{C}_1, \class{W}_1) , \ \ \ \class{M}_2 = (\class{C}_2, \class{W}_2) , \ \ \ \class{M}_3 = (\class{C}_3, \class{W}_3)$$ such that $\class{C}_2 , \class{C}_3 \subseteq  \class{C}_1$. If $\class{W}_3 \cap \class{C}_1 = \class{C}_2$ (or equivalently, $\class{W}_2 \cap \class{W}_3 = \class{W}_1$ and $\class{C}_2 \subseteq \class{W}_3$), then $\class{M}_2\backslash\class{M}_1$ is Quillen equivalent to $\class{M}_3$ and $\class{M}_3\backslash\class{M}_1$ is Quillen equivalent to $\class{M}_2$. In fact, we have a recollement as shown below.
\[
\begin{tikzpicture}[node distance=3.5cm]
\node (A) {$\mathcal{C}_2/\sim$};
\node (B) [right of=A] {$\mathcal{C}_1/\sim$};
\node (C) [right of=B] {$\class{A}/\class{W}_3$};
\draw[<-,bend left=40] (A.20) to node[above]{\small E$(\class{M}_3)$} (B.160);
\draw[->] (A) to node[above]{\small $I$} (B);
\draw[<-,bend right=40] (A.340) to node [below]{\small C$(\class{M}_2)$} (B.200);
\draw[<-,bend left] (B.20) to node[above]{\small $\lambda = \text{C}(\class{M}_3)$} (C.160);
\draw[->] (B) to node[above]{\small Q} (C);
\draw[<-,bend right] (B.340) to node [below]{\small $\rho = \text{E}(\class{M}_2) \circ \text{C}(\class{M}_1)$} (C.200);
\end{tikzpicture}
\]
Here, the functor $I$ is just inclusion while $Q$ is the quotient functor of (the dual of) Lemma~\ref{lemma-quotient map}. We point out that $\lambda$ has essential image $\class{C}_3/\sim$ while $\rho$ has essential image $(\class{W}_2 \cap \class{C}_1)/\sim$ and they provide an equivalence $\lambda : (\class{W}_2 \cap \class{C}_1)/\sim \, \longleftrightarrow  \class{C}_3/\sim \, : \rho$.

\end{theorem}

%%%%%%%%%%%%%%%%%%%%%%%%%%%%%%%%%%%%%%%%%%%%%%%%%%%%%%%%%%%%%%%%%%%%%%%%%%%%%%%%%%%%%%%%%%%%%%%%%%%
%%%%%%%%%%%%%%%%%%%%%%%%%%%%%%%%%%%%%%%%%%%%%%%%%%%%%%%%%%%%%%%%%%%%%%%%%%%%%%%%%%%%%%%%%%%%%%%%%%%

\section{Localizing cotorsion pairs and triples in Frobenius categories}\label{sec-localizing cot pairs in frobenius categories}

Recall that an exact category $\cat{A}$ is called a \emph{Frobenius category} if there are enough projective and injective objects and if these two classes of objects coincide. We will call these the \emph{projective-injective} objects. In this section we look at special cotorsion pairs in $\cat{A}$ which we call \emph{localizing} cotorsion pairs. Being interested in exact model structures and Hovey's correspondence with cotorsion pairs, we will only consider Frobenius categories which are weakly idempotent complete. For brevity, we will call such a category a \emph{WIC Frobenius category}. Indeed the next proposition points out that a Frobenius category $\cat{A}$ has an exact model structure if and only if $\cat{A}$ is a WIC Frobenius category.

\begin{proposition}\label{prop-frobenius cats are exact model cats iff WIC}
Let $\cat{A}$ be a Frobenius category. Then $\cat{A}$ has an exact model structure with the admissible monomorphisms (resp. admissible epimorphisms) as the cofibrations (resp. fibrations) and the projective-injective objects as the trivial objects if and only if $\cat{A}$ is weakly idempotent complete. Of course in this case $\Ho{\class{A}} = \class{A}/\sim$ is the stable category.
\end{proposition}

\begin{proof}
If $\cat{A}$ has such a model structure then since cofibrations satisfy the retract axiom we see that the admissible monomorphisms are closed under retracts. Proposition~2.4 of~\cite{gillespie-exact model structures} now says that $\cat{A}$ must be weakly idempotent complete. On the other hand, if $\cat{A}$ is weakly idempotent complete then Theorem~\ref{them-Hovey's theorem for WIC-exact categories} gives an exact model structure coming from the Hovey triple $(\class{A},\class{W},\class{A})$ where $\class{W}$ are the projective-injective objects.
\end{proof}

Since by definition, $\cat{A}$ has enough projective and injective objects we have from a WIC exact version of~\cite[Lemma~2.3]{gillespie-recollements} that a cotorsion pair $(\class{U},\class{V})$ is hereditary if and only if $\class{U}$ is resolving (or just syzygy closed) if and only if $\class{V}$ is coresolving (or just cosyzygy closed). The following proposition tells us more. In particular, $\class{U}$ is thick if and only if $\class{V}$ is thick.

 \begin{proposition}\label{prop-cot pairs in frobenius cats}
Let $(\class{U}, \class{V})$ be a cotorsion pair in a WIC Frobenius category $\cat{A}$. Then the following are equivalent.
\begin{enumerate}
\item $(\class{U}, \class{V})$ is hereditary with $\class{U}$ cosyzygy closed or $\class{V}$ syzygy closed.
\item $\class{U}$ is both syzygy and cosyzygy closed.
\item $\class{V}$ is both syzygy and cosyzygy closed.
\item $\class{U}$ is thick.
\item $\class{V}$ is thick.
\end{enumerate}
Moreover, if $(\class{U},\class{V})$ is complete then the conditions above are also equivalent to:
\begin{enumerate}
\setcounter{enumi}{5}
\item $(\class{U}, \class{V})$ is an injective cotorsion pair. That is, $\class{U}$ is thick and $\class{U} \cap \class{V}$ is the class of injectives.
\item $(\class{U}, \class{V})$ is a projective cotorsion pair. That is, $\class{V}$ is thick and $\class{U} \cap \class{V}$ is the class of projectives.
\end{enumerate}

\end{proposition}

\begin{proof}
Note that the projective-injective objects are automatically in both $\class{U}$ and $\class{V}$. Again we will cite WIC exact versions of basic results appearing in~\cite{gillespie-recollements}. We have (2) implies (1) by~\cite[Lemma~2.3]{gillespie-recollements}. We will show (1) implies (4), in particular, hereditary together with $\class{V}$ syzygy closed implies $\class{U}$ is thick. Indeed given any $V \in \class{V}$, we may use enough projectives to find an exact $K \rightarrowtail P \twoheadrightarrow V$ where $P$ is projective-injective. By assumption we have $K \in \class{V}$. So it follows from~\cite[Lemma~3.5]{gillespie-recollements} that $\class{U}$ is thick, proving (4). Now (4) implies (2) trivially. So we have shown (2) implies (1) implies (4) implies (2).

Similarly, we have (3) implies (1) which implies (5) (here, the dual of~\cite[Lemma~3.5]{gillespie-recollements} applies to show $\class{U}$ cosyzygy closed implies $\class{V}$ thick) which implies (3). Together this proves the equivalence of (1)--(5).

Next assume that $(\class{U}, \class{V})$ is complete. Then we see from Proposition~\ref{prop-characterizations of injective cotorsion pairs} that $(\class{U}, \class{V})$ is an injective cotorsion pair if and only if $\class{U}$ is thick. Similarly, the dual \cite[Proposition~3.7~(3)]{gillespie-recollements} shows that $(\class{U}, \class{V})$ is a projective cotorsion pair if and only if $\class{V}$ is thick.

\end{proof}

This leads us to make the following convenient definition.

\begin{definition}\label{def-localizing cot pair in frobenius}
 We call a complete cotorsion pair $(\class{U}, \class{V})$ in a WIC Frobenius category a \emph{localizing cotorsion pair} if it satisfies any of the equivalent conditions of Proposition~\ref{prop-cot pairs in frobenius cats}. We note that in this case $\class{U} \cap \class{V}$ equals the class of projective-injective objects.

\end{definition}

\begin{remark}
Note that Salce's argument applies in this setting, so a sufficient condition for $(\class{U}, \class{V})$ to be complete is that it has either enough projectives or enough injectives.
\end{remark}

A localizing cotorsion pair in $\class{A}$ is equivalent to a Bousfield localizing pair in $\class{A}/\sim$ through a simple correspondence $(\class{U}, \class{V}) \leftrightarrow (\class{U}/\sim \,, \class{V}/\sim)$. This follows from~\cite[Proposition~3.8]{saorin-stovicek}, where the definition of a Bousfield localizing pair also appears. Alternatively, the correspondence is a special case of a WIC exact category version of~\cite[Proposition~4.15]{gillespie-recollements}. Again, our focus here is on the cotorsion pairs as they are equivalent to exact model structures and give us a general framework to describe localizations via the ground category $\class{A}$. Indeed any localizing cotorsion pair in a WIC Frobenius category has associated to it \emph{two} exact model structures on $\cat{A}$. An injective one, killing the objects of $\class{U}$ and a projective one killing the objects of $\class{V}$.

\begin{corollary}\label{cor-two model strucs from a localizing cot pair}
Let $\cat{A}$ be a WIC Frobenius category and let $\class{M} = (\class{U},\class{V})$ be a localizing cotorsion pair. Then  $\class{M} = (\class{U},\class{V})$ gives rise to two model structures on $\cat{A}$. The first is the Hovey triple $\class{M}^i = (\class{A}, \class{U}, \class{V})$ and we call it the \emph{injective model structure induced by $(\class{U},\class{V})$}. The second is the Hovey triple $\class{M}^p = (\class{U}, \class{V}, \class{A})$ and we call it the \emph{projective model structure induced by $(\class{U},\class{V})$}.
\end{corollary}

Note that the canonical cotorsion pair $\class{M} = (\class{W},\class{A})$,  where $\class{W}$ is the class of projective-injective objects, is localizing. Viewing it as the categorical projective cotorsion pair $\class{M}^p = (\class{W},\class{A})$ corresponds to the trivial model $(\class{W},\class{A},\class{A})$. But it may also be viewed as the categorical Gorenstein injective cotorsion pair $\class{M}^i = (\class{W},\class{A})$ which corresponds to $(\class{A},\class{W},\class{A})$ and is a model for the stable category $\class{A}/\sim$. For any other localizing cotorsion pair $\class{N} = (\class{U},\class{V})$ we note that the right localization satisfies $\class{M}^i/\class{N}^i = \class{N}^p$. Similar observations apply to the other canonical cotorsion pair $(\class{A},\class{W})$ and left localization.

\subsection{Localizing cotorsion triples and recollements}
We wish to show now that Theorems~\ref{them-recollements in krause form} and~\ref{them-recollements theorem projective version} each recover the recollement associated to a torsion triple in the stable category $\cat{A}/\sim$ of a WIC Frobenius category $\cat{A}$. Recall that a torsion triple in a triangulated category $\class{T}$ is a triple $(\class{X},\class{Y},\class{Z})$ of thick subcategories of $\class{T}$ for which $(\class{X},\class{Y})$ and $(\class{Y},\class{Z})$ are each torsion pairs. See~\cite{beligiannis-reiten}. Torsion triples correspond to recollements in the way described in~\cite[page~25]{gillespie-recollements}. To lift this to the level of model structures, lets call a triple of classes $(\class{X},\class{Y},\class{Z})$ in a WIC Frobenius category $\cat{A}$ a \textbf{localizing cotorsion triple} if $(\class{X},\class{Y})$ and $(\class{Y},\class{Z})$ are localizing cotorsion pairs in $\cat{A}$.

\begin{corollary}\label{cor-recollements for WIC Frobenius cats}
Let $(\class{X},\class{Y},\class{Z})$ be classes in a WIC Frobenius category $\cat{A}$. Then $(\class{X},\class{Y},\class{Z})$ is a localizing cotorsion triple in $\cat{A}$ if and only if
$(\class{X}/\sim,\class{Y}/\sim,\class{Z}/\sim)$ is a torsion triple in the stable category $\cat{A}/\sim \,$. In this case, we have equivalences of triangulated categories
$$\class{X}/\sim \ \ \cong \ \class{A}/\class{Y} \ \ \cong \ \class{Z}/\sim \,.$$ Moreover, Theorems~\ref{them-recollements in krause form} and~\ref{them-recollements theorem projective version} each recover the expected recollement in the following way.
\begin{enumerate}
\item Let $\class{M}_1 = (\class{W},\class{A})$ be the canonical localizing cotorsion pair so that $\class{M}^i_1 = (\class{W},\class{A})$ is a model for the stable category $\cat{A}/\sim \,$. Taking in Theorem~\ref{them-recollements in krause form} the injective cotorsion pairs to be $$\class{M}^i_1 = (\class{W}, \class{A}) , \ \ \ \class{M}^i_2 = (\class{X}, \class{Y}) , \ \ \ \class{M}^i_3 = (\class{Y}, \class{Z})$$
yields a recollement as below.
\[
\begin{tikzpicture}[node distance=3.5cm]
\node (A) {$\mathcal{Y}/\sim$};
\node (B) [right of=A] {$\mathcal{A}/\sim$};
\node (C) [right of=B] {$\class{A}/\class{Y}$};
\draw[<-,bend left=40] (A.20) to node[above]{\small E$(\class{X},\class{Y})$} (B.160);
\draw[->] (A) to node[above]{\small $I$} (B);
\draw[<-,bend right=40] (A.340) to node [below]{\small C$(\class{Y},\class{Z})$} (B.200);
\draw[<-,bend left] (B.20) to node[above]{\small $\lambda = \text{C}(\class{X},\class{Y})$} (C.160);
\draw[->] (B) to node[above]{\small Q} (C);
\draw[<-,bend right] (B.340) to node [below]{\small $\rho = \text{E}(\class{Y},\class{Z})$} (C.200);
\end{tikzpicture}
\]
The functor $I$ is inclusion while $Q$ is the quotient functor of Lemma~\ref{lemma-quotient map}. We point out that $\lambda$ has essential image $\class{X}/\sim$ while  $\rho$ has essential image $\class{Z}/\sim$ and they provide an equivalence $\lambda : \class{Z}/\sim \, \longleftrightarrow  \class{X}/\sim \, : \rho$.

\item On the other hand, consider the canonical cotorsion pair $\class{M}_1 = (\class{A},\class{W})$. Note $\class{M}^p_1 = (\class{A},\class{W})$ is again the same model for the stable category $\cat{A}/\sim \,$. Taking in Theorem~\ref{them-recollements theorem projective version} the projective cotorsion pairs to be $$\class{M}^p_1 = (\class{A}, \class{W}) , \ \ \ \class{M}^p_2 = (\class{Y}, \class{Z}) , \ \ \ \class{M}^p_3 = (\class{X}, \class{Y})$$
yields the exact same recollement diagram as the one above.

\end{enumerate}

\end{corollary}

\begin{proof}
It follows from~\cite[Proposition~3.8]{saorin-stovicek} that $(\class{X},\class{Y},\class{Z})$ is a localizing cotorsion triple in $\cat{A}$ if and only if
$(\class{X}/\sim,\class{Y}/\sim,\class{Z}/\sim)$ is a torsion triple in $\cat{A}/\sim \,$. Note that $\class{M}^i_1/\class{M}^i_2 = \class{M}^p_2$ has the same trivial objects, $\class{Y}$, as $\class{M}^i_3$. Also, $\class{M}^i_1/\class{M}^i_3 = \class{M}^p_3$ has the same cofibrant-fibrant objects, again $\class{Y}$, as $\class{M}^i_2$.

For the projective case, note that we have swapped the role of $\class{M}_2$ and $\class{M}_3$.

\end{proof}

\section{Localizing cotorsion pairs and triples in $\cha{A}_{dw}$}\label{sec-localizing cot pairs in ch_dw}

Recollement situations involving homotopy categories of chain complexes are of particular interest and will be the focus of the rest of this paper. In this section we wish to pursue further the notion of localizing cotorsion pairs, but specialized to the general setting of chain complexes over an additive category $\cat{A}$.
Since we are interested in the correspondence between model structures and cotorsion pairs we now freely assume that $\cat{A}$ is weakly idempotent complete. This makes $\cha{A}_{dw}$ WIC Frobenius (see Subsection~\ref{subsec-chain complexes form a Frobenius cat}). We get the following characterization of hereditary cotorsion pairs in $\cha{A}_{dw}$ by Lemma~\ref{lemma-suspensions are cosyzygies}.

\begin{proposition}\label{prop-hereditary dw cotorsion pairs}
Let $(\class{X}, \class{Y})$ be a cotorsion pair in the WIC Frobenius $\cha{A}_{dw}$. Then the following are equivalent.
\begin{enumerate}
\item $(\class{X}, \class{Y})$ is a hereditary cotorsion pair.
\item $\class{X}$ is closed under negative suspensions.
\item $\class{Y}$ is closed under positive suspensions.

\end{enumerate}

\begin{proof}
Any cotorsion pair in a WIC exact category will have each class closed under finite direct sums and direct summands. So $(\class{X},\class{Y})$ being a cotorsion pair in $\cha{A}_{dw}$ implies $\class{X}$ and $\class{Y}$ each must be closed under direct sums and direct summands. Clearly each class also contains the contractible complexes, which are the projective-injective objects. So the result follows from combining Lemma~\ref{lemma-suspensions are cosyzygies} with a WIC exact version of~\cite[Lemma~2.3]{gillespie-recollements}. The content of this cited Lemma is standard. See the paragraph before Proposition~\ref{prop-cot pairs in frobenius cats}.
\end{proof}
\end{proposition}

Next, we make precise the fact that orthogonality of suspension closed classes with respect to $\Ext^1_{dw}$ is the same as orthogonality with respect to null homotopic maps. We state this in terms of the hom-complex $\homcomplex$.

\begin{definition}\label{def-Homcomplex pair}
Let $(\class{X},\class{Y})$ be a pair of classes of chain complexes in the WIC Frobenius $\cha{A}_{dw}$. We say that $(\class{X},\class{Y})$ is a $\homcomplex$-pair if $\class{X}$ and $\class{Y}$ are orthogonal with respect to $\homcomplex$-exactness. More precisely, this means all three of the following hold:
\begin{enumerate}
    \item $\homcomplex(X,Y)$ is exact for all $X \in \class{X}$ and $Y \in \class{Y}$.

    \item If $\homcomplex(X,Y)$ is exact for all $X \in \class{X}$, then $Y
    \in \class{Y}$.

    \item If $\homcomplex(X,Y)$ is exact for all $Y \in \class{Y}$, then $X \in
    \class{X}$.
\end{enumerate}
It is immediate from Lemma~\ref{lemma-homcomplex-basic-lemma} that $\class{X}$ and $\class{Y}$ are suspension closed and that in the above three conditions we can replace the statement ``$\homcomplex(X,Y)$ is exact'' with ``every chain map $f : X \xrightarrow{} Y$ is null homotopic''.

\end{definition}

\begin{proposition}\label{prop-HOM pairs classification}
Let $(\class{X}, \class{Y})$ be a cotorsion pair in the WIC Frobenius $\cha{A}_{dw}$. Then the following are equivalent.
\begin{enumerate}
\item $(\class{X}, \class{Y})$ is a $\homcomplex$-pair.
\item $\class{X}$ is closed under suspensions.
\item $\class{Y}$ is closed under suspensions.
\item $(\class{X}, \class{Y})$ is hereditary with $\class{X}$ cosyzygy closed or $\class{Y}$ syzygy closed.
\item $\class{X}$ is both syzygy and cosyzygy closed.
\item $\class{Y}$ is both syzygy and cosyzygy closed.
\item $\class{X}$ is thick.
\item $\class{Y}$ is thick.
\end{enumerate}

\end{proposition}

\begin{proof}
Conditions (4) through (8) are precisely conditions (1) through (5) of Proposition~\ref{prop-cot pairs in frobenius cats} and so we already know (4) through (8) are equivalent. Moreover, we know (2) and (5) are equivalent by Lemma~\ref{lemma-suspensions are cosyzygies}, as are (3) and (6). Hence we already know (2) through (8) are all equivalent.

So we only need to show that (1) implies any of the other conditions and that vice versa, any one of the other conditions implies (1). First we show (1) implies (6).
For this, note that for any $X$, the functor $\homcomplex(X,-) : \cha{A}_{dw} \xrightarrow{} \textnormal{Ch}(\Z)$ is exact. That is, it carries degreewise split
exact sequences to short exact sequences. Indeed, if $W \rightarrowtail Y \twoheadrightarrow Z$ is degreewise split, then for any
integers $n,k$ we see
$$0 \xrightarrow{} \Hom_{\cat{A}}(X_k,W_{n+k}) \xrightarrow{} \Hom_{\cat{A}}(X_k,Y_{n+k}) \xrightarrow{} \Hom_{\cat{A}}(X_k,Z_{n+k}) \xrightarrow{} 0$$ is a (split) short exact sequence of abelian groups. Therefore, for a fixed $n$,
the product
$$0 \xrightarrow{} \prod_{k \in \Z} \Hom_{\cat{A}}(X_{k},W_{k+n}) \xrightarrow{}
\prod_{k \in \Z} \Hom_{\cat{A}}(X_{k},Y_{k+n}) \prod_{k \in \Z}
\xrightarrow{} \Hom_{\cat{A}}(X_{k},Z_{k+n}) \xrightarrow{} 0$$ is exact. But this is degree $n$ after applying $\homcomplex(X,-)$ to $W \rightarrowtail Y \twoheadrightarrow Z$. So we have the exactness of $0 \xrightarrow{} \homcomplex(X,W) \xrightarrow{}
\homcomplex(X,Y) \xrightarrow{} \homcomplex(X,Z) \xrightarrow{}
0$. It now follows from the long exact sequence in homology that if any two out of three terms in a degreewise split exact sequence $W \rightarrowtail Y \twoheadrightarrow Z$ are in $\class{Y}$, then so is the third. In particular, $\class{Y}$ must be both syzygy and cosyzygy closed whenever $(\class{X},\class{Y})$ is a $\homcomplex$-pair.

Finally, we suppose (3) is true and we show that $(\class{X},\class{Y})$ is a $\homcomplex$-pair. For this, let $X \in \class{X}$ and $Y \in \class{Y}$. Then $\Sigma^{-n-1}Y \in \class{Y}$ too and using Lemma~\ref{lemma-homcomplex-basic-lemma} we get $H_n\homcomplex(X,Y) = \Ext^1_{dw}(X,\Sigma^{-n-1}Y) = 0$ since $(\class{X},\class{Y})$ is a cotorsion pair. So $\homcomplex(X,Y)$ is exact. This verifies the first condition in the definition of $\homcomplex$-pair. For the second condition, suppose $\homcomplex(X,Y)$ is exact for all $X \in \class{X}$. Then in particular, $H_{-1}(X,Y) = 0$.  But then $0 = H_{-1}(X,Y) = \Ext^1_{dw}(X,Y)$. Since $(\class{X},\class{Y})$ is a cotorsion pair we get $Y \in \class{Y}$. This verifies the second condition and the third condition is similar.

\end{proof}

\begin{corollary}\label{cor-classification of localizing cot pairs in Ch(R)dw}
Let $(\class{X}, \class{Y})$ be a complete cotorsion pair in $\cha{A}_{dw}$. Then $(\class{X}, \class{Y})$ is a localizing cotorsion pair if and only if it satisfies the equivalent properties in Proposition~\ref{prop-HOM pairs classification}. In particular, note that if $\class{Y}$ is a suspension closed class of complexes, and $(\class{X},\class{Y})$ and $(\class{Y},\class{Z})$ are each complete cotorsion pairs in $\cha{A}_{dw}$, then $(\class{X},\class{Y},\class{Z})$ is a localizing cotorsion triple and immediately yields a recollement as in Corollary~\ref{cor-recollements for WIC Frobenius cats}.

\end{corollary}

\section{The K-injective and K-projective model structures}\label{sec-K-injective and K-projective model structures}

Let $\cat{G}$ be a Grothendieck category. In this section we use the theory of exact model structures to construct a model structure on $\cha{G}$ whose homotopy category is the derived category $\class{D}(\cat{G})$ and whose fibrant objects are the K-injective complexes as defined by Spaltenstein in~\cite{spaltenstein}. We apply this to immediately obtain, in the affine case, the classical recollement of Verdier as an easy corollary of Theorem~\ref{them-recollements in krause form}. But rather than give a direct proof, we instead opt to prove a more general phenomenon which we will have the opportunity to apply again in Section~\ref{sec-Murfet and Neeman models and recollements}.

For now, let $R$ be a ring and recall the following definition from~\cite{spaltenstein}.

\begin{definition}\label{def-K-injective and K-projective complexes}
We say a complex $K$ is \emph{K-injective} if the complex $\homcomplex(E,K)$ is exact whenever $E$ is an exact complex. On the other hand we say $K$ is \emph{K-projective} if $\homcomplex(K,E)$ is exact whenever $E$ is exact.

\end{definition}

By Lemma~\ref{lemma-homcomplex-basic-lemma} we get that a complex $K$ is $K$-injective if and only if every chain map $f : E \xrightarrow{} K$ is null homotopic whenever $E$ is exact. It is immediate then that an exact K-injective complex is contractible. Conversely, a contractible complex is both exact and K-injective. By definition, a complex $I$ is DG-injective if and only if each $I_n$ is injective and $I$ is K-injective.

We now generalize K-injective and K-projective complexes as follows. Assume $\cat{A}$ is any WIC exact category with enough injectives. Recall from Lemma~\ref{lemma-ch(A)} that this makes $\cha{A}$ also WIC exact with enough injectives. So it is fertile ground to discuss exact model structures coming from injective cotorsion pairs. Let $(\class{U},\class{F})$ be such an injective cotorsion pair in $\cha{A}$ and assume $\class{F}$ is contained within the class of degreewise injective complexes.  Examples of such $(\class{U},\class{F})$ that the author has in mind include those obtained by taking $\class{F}$ to be one of the following classes of complexes of modules over a ring $R$: All complexes of injectives, the exact complexes of injectives, the DG-injectives, and the (exact) AC-acyclic complexes of injectives. See~\cite{bravo-gillespie-hovey} for details on these model structures on $\ch$. We will show that $(\class{U},\class{F})$ generates a Quillen equivalent localizing cotorsion pair $(\class{U},K\class{F})$ in $\cha{A}_{dw}$ where $K\class{F}$ is the class of all complexes which are isomorphic, in the homotopy category $K(\cat{A})$, to some complex in $\class{F}$. We think of these as ``K-versions'' or ``Spaltenstein versions'' of the model structure $(\class{U},\class{F})$. Indeed when $\class{F}$ is the class of DG-injective complexes then $K\class{F}$ is exactly the class of K-injective complexes. Throughout we give only injective versions of the general results but there are obvious projective versions too whenever $\cat{A}$ is WIC exact with enough projectives.

\begin{definition}\label{def-K-F-injective and K-Fprojective complexes}
Let $(\class{U},\class{F})$ be an injective cotorsion pair in $\cha{A}$ with $\class{F}$ contained within the class of degreewise injective complexes. We will call a complex $F \in \class{F}$ an \emph{$\class{F}$-injective} complex and we will say a complex $X$ is \emph{$K\class{F}$-injective} if the complex $\homcomplex(U,X)$ is exact whenever $U \in \class{U}$. We denote the class of all $K\class{F}$-injective complexes by $K\class{F}$.

\end{definition}

\begin{remark}
Since $\class{U}$ is thick and contains the class $\class{W}$ of contractible complexes it is automatic from Lemma~\ref{lemma-suspensions are cosyzygies} that $\class{U}$ is closed under suspensions. So by Lemma~\ref{lemma-homcomplex-basic-lemma} we get that a complex $X$ is $K\class{F}$-injective if and only if every chain map $f :U \xrightarrow{} X$, with $U \in \class{U}$, is null homotopic. From this it is easy to see that $\class{U} \cap K\class{F} = \class{W}$. Note that a complex $F$ is $\class{F}$-injective if and only if $F$ is $K\class{F}$-injective and each $F_n$ is injective.
\end{remark}

\begin{theorem}\label{them-generalized K-injective models}
Let $\cat{A}$ be a WIC exact category with enough injectives. Assume we have an injective cotorsion pair $(\class{U},\class{F})$ in $\cha{A}$ with $\class{F}$ contained within the class of degreewise injective complexes.
Then $(\class{U},K\class{F})$ is a localizing cotorsion pair in $\cha{A}_{dw}$. Its induced injective model structure $(\class{A}, \class{U}, K\class{F})$ we will call the \emph{$K\class{F}$-injective model structure on $\cha{A}$}. We note the following properties of the $K\class{F}$-injective model structure on $\cha{A}$.
\begin{enumerate}
\item The cofibrations are the degreewise split monos.
\item The trivial cofibrations are the degreewise split monos with cokernels in $\class{U}$.
\item The fibrations are the degreewise split epis with $K\class{F}$-injective kernel.
\item The trivial fibrations are the degreewise split epis with contractible kernel.
\item The weak equivalences are the maps which factor as a trivial cofibration followed by a trivial fibration.
\item The identity functor on $\cha{A}$ is a (right) Quillen equivalence from the $\class{F}$-injective model structure on $\cha{A}$ to the $K\class{F}$-injective model structure.
\end{enumerate}
\end{theorem}

\begin{proof}
First, we have that $\class{U} = \leftperp{\class{F}}$ and note that this leftperp is the same whether it is taken in either $\cha{A}$ or $\cha{A}_{dw}$ because we are assuming $\class{F}$ consists of complexes which are injective in each degree. Next the rightperp $\rightperp{\class{U}}$ taken in $\cha{A}_{dw}$ coincides with $K\class{F}$ by its definition along with Lemma~\ref{lemma-homcomplex-basic-lemma} and the fact that the classes are closed under suspensions. So $(\class{U},K\class{F})$ is a cotorsion pair in $\cha{A}_{dw}$.
Now we prove completeness. Let $X$ be any chain complex. Since $(\class{U},\class{F})$ is a complete cotorsion pair in $\cha{A}$ we can find an admissible short exact sequence $F \rightarrowtail U \twoheadrightarrow X$ where $U \in \class{U}$ and $F$ is $\class{F}$-injective.  But since $F$ is a complex of injectives this sequence is automatically an admissible short exact sequence in $\cha{A}_{dw}$. Since $F$ is also $K\class{F}$-injective we have shown that $(\class{U}, K\class{I})$ has enough projectives in $\cha{A}_{dw}$. Completeness now follows from the Remark following Definition~\ref{def-localizing cot pair in frobenius}. This completes the proof that $(\class{U},K\class{I})$ is a localizing cotorsion pair in $\cha{A}_{dw}$.

Now properties (1)--(5) are now immediate from the definition of an exact model structure. We show (6). First, we claim the identity functor from the $\class{F}$-injective model structure to the $K\class{F}$-injective model structure is a right Quillen functor. A fibration in the $\class{F}$-injective structure is a (necessarily) degreewise split epimorphism with $\class{F}$-injective kernel. Since $\class{F}$-injective complexes are $K\class{F}$-injective we conclude the identity preserves fibrations. Similarly, a trivial fibration in the $\class{F}$-injective model structures is a (necessarily) degreewise split epimorphism with injective kernel. Since injective complexes are contractible we conclude the identity preserves trivial fibrations. So the identity is a right Quillen functor. By the definition of a Quillen equivalence we will be done if we can show that a map $f$ is a weak equivalence in the $\class{F}$-injective model if and only if it is a weak equivalence in the $K\class{F}$-injective model. First suppose that $f$ is a weak equivalence in the $K\class{F}$-injective model structure. Then by definition, it factors as $f = pi$ where $i$ is a degreewise split monomorphism with cokernel in $\class{U}$ and $p$ is a degreewise split epimorphism with contractible kernel. But each contractible complex is in $\class{U}$. So it follows from Lemma~\ref{lemma-characterization of weak equivalences in exact categories} that $f$ is a weak equivalence in the $\class{F}$-injective model structure. On the other hand, say $f$ is a weak equivalence in the $\class{F}$-injective model. Using the factorization axiom in the $K\class{F}$-injective model structure first factor it as $f = pi$ where $i$ is a trivial cofibration and $p$ is a fibration. So $i$ is a degreewise split monomorphism with cokernel in $\class{U}$ and $p$ a degreewise split monomorphism with $K\class{F}$-injective kernel. But then $i$ is also a trivial cofibration in the $\class{F}$-injective model structure. So by the two out of three axiom, $p$ must also be a weak equivalence in the $\class{F}$-injective model structure. This means that $\ker{p}$ is injective and it follows that $p$ and therefore $f$ are weak equivalences in the $K\class{F}$-injective model structure. This completes the proof of (6).

\end{proof}

The following gives a more detailed characterization of the $K\class{F}$-injective complexes, essentially saying that they make up the isomorphic closure of the class of $\class{F}$-injectives in the homotopy category $K(\cat{A})$.
Recall that $X$ and $Y$ are isomorphic in $K(\cat{A})$ if they are chain homotopy equivalent which means that there are chain maps $f : X \xrightarrow{} Y$ and $g : Y \xrightarrow{} X$ such that $gf \sim 1_X$ and $fg \sim 1_Y$.

\begin{proposition}\label{prop-characterization of KF-injective complexes}
Continuing with Theorem~\ref{them-generalized K-injective models}, let $X$ be a chain complex. Then the following are equivalent.
\begin{enumerate}
\item $X$ is $K\class{F}$-injective.
\item There exists an $\class{F}$-injective complex $F$ and contractible complexes $W_1$ and $W_2$ such that $X \oplus W_2 \cong F \oplus W_1$.
\item $X$ is chain homotopy equivalent to some $\class{F}$-injective complex $F$.
\end{enumerate}
\end{proposition}

\begin{proof}
First we have that (2) and (3) are equivalent for formal reasons concerning Frobenius categories. Indeed in any WIC Frobenius category $\cat{A}$, two objects $A,B \in \class{A}$ are isomorphic in $\class{A}/\sim$ if and only if there are projective-injective objects $W_1,W_2 \in \class{A}$ and an $\class{A}$-isomorphism $A \oplus W_1 \cong B \oplus W_2$. (See~\cite[Lemma~3.1]{saorin-stovicek} or~\cite[Lemma~1.4.3~(1)]{becker}.) So we get the result from this, along with Proposition~\ref{prop-Ch(A)_dw is Frobenius}.

Now suppose (1), so $X$ is $K\class{F}$-injective and write a degreewise split short exact sequence $X \rightarrowtail W_1 \twoheadrightarrow K$ where $W_1$ is contractible. Since $K\class{F}$ is thick in $\cha{A}_{dw}$ we see $K$ must also be $K\class{F}$-injective. Now by hypothesis we may use that $(\class{U}, \class{F})$ is a complete cotorsion pair in $\cha{A}$ to get an admissible short exact sequence $F \rightarrowtail W_2 \twoheadrightarrow K$ with $F \in \class{F}$ and $W_2 \in \class{U}$. Note that the sequence must be degreewise split since each $F_n$ is injective. Since $W_2$ is a degreewise split extension of two $K\class{F}$-injective complexes we get that $W_2 \in \class{U} \cap K\class{F}$, and so is contractible. Now (2) will follow from the dual of Schanuel's Lemma~\ref{lemma-schanuel's lemma}. That is, construct the pullback diagram below in $\cha{A}_{dw}$ and note that we must have $X \oplus W_2 \cong P \cong F \oplus W_1$.
$$\begin{tikzcd}
X \arrow[tail]{r} \arrow[equal]{d} & W_1 \arrow[two heads]{r}  & K  \\
X \arrow[tail]{r} & P \arrow[two heads]{r} \arrow[two heads]{u} & W_2 \arrow[two heads]{u} \\
& F \arrow[equal]{r} \arrow[tail]{u} & F \arrow[tail]{u}
\end{tikzcd}$$
The converse (2) implies (1) is true since $K\class{F}$ is thick in $\cha{A}_{dw}$ and contains both the $\class{F}$-injective complexes and the contractible complexes.
\end{proof}

\subsection{The K-injective and K-projective model structures}
Let $\cat{G}$ be any Grothendieck category. $\cat{G}$ is an exact category with its abelian structure. This is the setting for the remainder of this section and we now set the following notation which we also keep for the remainder of the section. We let $\class{W}$ denote the class of all contractible complexes, $\class{E}$ denote the class of all exact complexes, $K\class{I}$ the class of all K-injective complexes, $K\class{P}$ the class of all K-projective complexes, and $\class{A}$ the class of all complexes. The following corollary is immediate from Theorem~\ref{them-generalized K-injective models} and Proposition~\ref{prop-characterization of KF-injective complexes} using the known fact that the DG-injective cotorsion pair $(\class{E},\dgclass{I})$ is an injective cotorsion pair.  See for example, \cite[Theorem~4.20]{saorin-stovicek} or~\cite[Example~5.1]{gillespie-degreewise-model-strucs}.

\begin{corollary}\label{cor-K-injective model for derived category}
$(\class{E},K\class{I})$ is a localizing cotorsion pair in $\cha{G}_{dw}$. Its induced injective model structure $(\class{A}, \class{E}, K\class{I})$ on $\cha{G}_{dw}$ we call the \emph{K-injective model structure on $\cha{G}$}. We note the following properties of the K-injective model structure on $\cha{G}$.
\begin{enumerate}
\item The cofibrations are the degreewise split monos.
\item The trivial cofibrations are the degreewise split monos with exact cokernels.
\item The fibrations are the degreewise split epis with K-injective kernel.
\item The trivial fibrations are the degreewise split epis with contractible kernel.
\item The weak equivalences are the homology isomorphisms.
\item The identity functor on $\cha{G}$ is a (right) Quillen equivalence from the usual injective model structure on $\cha{G}$ for the derived category $\cat{D}(\cat{G})$ to the K-injective model structure.
\end{enumerate}
The following statements are equivalent and characterize the fibrant complexes $X$.
\begin{itemize}
\item $X$ is K-injective.
\item There exists a DG-injective complex $I$ and contractible complexes $W_1$ and $W_2$ such that $X \oplus W_1 \cong I \oplus W_2$ in $\cha{G}$.
\item $X$ is chain homotopy equivalent to some DG-injective complex $I$.
\end{itemize}
\end{corollary}

For the dual we need to assume that $\cat{G}$ has a projective generator. This forces $\cat{G}$ to have enough projectives. In this case, it follows that the dual DG-projective cotorsion pair $(\dgclass{P},\class{E})$ is a projective cotorsion pair. In particular, Proposition~3.8 of~\cite{gillespie-quasi-coherent} can be used to show that it is complete.

\begin{corollary}\label{cor-K-projective model for derived category}
Assume $\cat{G}$ has a projective generator. Then
$(K\class{P},\class{E})$ is a localizing cotorsion pair in $\cha{G}_{dw}$. Its induced projective model structure $(K\class{P}, \class{E}, \class{A})$ on $\cha{G}_{dw}$ we call the \emph{K-projective model structure on $\cha{G}$}. We note the following properties of the K-projective model structure on $\cha{G}$.
\begin{enumerate}
\item The fibrations are the degreewise split epis.
\item The trivial fibrations are the degreewise split epis with exact kernels.
\item The cofibrations are the degreewise split monos with K-projective cokernel.
\item The trivial cofs are the degreewise split monos with contractible cokernel.
\item The weak equivalences are the homology isomorphisms.
\item The identity functor on $\cha{G}$ is a (left) Quillen equivalence from the usual projective model structure on $\cha{G}$ for $\cat{D}(\cat{G})$ to the K-projective model structure.
\end{enumerate}
The following are equivalent and characterize the cofibrant complexes $X$.
\begin{itemize}
\item $X$ is K-projective.
\item There exists a DG-projective complex $P$ and contractible complexes $W_1$ and $W_2$ such that $X \oplus W_1 \cong P \oplus W_2$ in $\cha{G}$.
\item $X$ is chain homotopy equivalent to some DG-projective complex $P$.
\end{itemize}
\end{corollary}

%%%%%%%%%%%%%%%%%%%%%%%%%%%%%%%%%%%%%%%%%%%%%%%%%%%%%%%%%%%%%%%%%%%%%%%%%%%%%%%%%%%%%%%%
\subsection{The Verdier recollement via model categories}
As in Corollary~\ref{cor-K-projective model for derived category}, assume the Grothendieck category $\cat{G}$ has a projective generator.
Having constructed the K-injective and K-projective model structures on $\cha{G}$ using cotorsion pairs, we immediately obtain as a corollary the recollement of Verdier. This reflects the usual passage from $K(\cat{G})$ to $\class{D}(\cat{G})$ by taking the Verdier quotient with respect to the thick class of exact complexes in $K(\cat{G})$. But it is interesting to see this phenomenon come from model structures using just cotorsion pairs, and no prior construction of $K(\cat{G})$ or $\class{D}(\cat{G})$. Recall that we are still using the notation: $\class{W}$ = contractible complexes, $\class{E}$ = exact complexes, $K\class{I}$ = K-injective complexes, $K\class{P}$ = K-projective complexes, $\class{A}$ = class of all complexes.

\begin{corollary}\label{cor-Verdier recollement via cotorsion pairs}
Let $\cat{G}$ be a Grothendieck category with a projective generator. Then $(K\class{P}, \class{E}, K\class{I})$ is a localizing cotorsion triple in $\cha{G}_{dw}$. Corollary~\ref{cor-recollements for WIC Frobenius cats} tells us
we have equivalences of triangulated categories $$K\class{P}/\sim \ \cong \ \class{D}(\cat{G}) \ \cong \ K\class{I}/\sim $$
and that we have a recollement as shown.
\[
\begin{tikzpicture}[node distance=3.5cm]
\node (A) {$\mathcal{E}/\sim$};
\node (B) [right of=A] {$K(\cat{G})$};
\node (C) [right of=B] {$\class{D}(\cat{G})$};
\draw[<-,bend left=40] (A.20) to node[above]{\small E$(K\class{P},\class{E})$} (B.160);
\draw[->] (A) to node[above]{\small $I$} (B);
\draw[<-,bend right=40] (A.340) to node [below]{\small C$(\class{E},K\class{I})$} (B.200);
\draw[<-,bend left] (B.20) to node[above]{\small $\lambda = \text{C}(K\class{P},\class{E})$} (C.160);
\draw[->] (B) to node[above]{\small Q} (C);
\draw[<-,bend right] (B.340) to node [below]{\small $\rho = \text{E}(\class{E},K\class{I})$} (C.200);
\end{tikzpicture}
\]

\end{corollary}

\begin{remark}\label{remark-NOT cofibrantly generated}
The K-injective and K-projective model structures in $\cha{G}$ are not cofibrantly generated in general. In fact even for most rings they are not. This is due to the fact that the exact category $\ch_{dw}$ typically won't have a set of generators. See~\cite[Remark~1.6]{saorin-stovicek}.

\end{remark}

\section{Model structures and recollements on complexes of flat modules}\label{sec-Murfet and Neeman models and recollements}

In this section we give model category interpretations of recollements due to Neeman and Murfet from~\cite{neeman} and~\cite{murfet}. Some of the model structures we construct will require that $R$ be a ring while others hold for more general categories of quasi-coherent sheaves. Since we are interested in how these models are interlaced through recollement situations we will let $R$ be a ring throughout, as this is needed for all the recollements except the last one. But the interested reader will have no problems noticing that all the statements that don't require enough projectives hold in more generality. The reason for this is that all the difficulties have been front-loaded to cotorsion pairs and follow formally from completeness of the flat cotorsion pair. We end this section by returning to this point. The first main point we are making in the presentations here is that once the ``correct'' model structures are found, the recollement situations from~\cite{neeman} and~\cite{murfet} follow at once. The second main point is that the correct model structures are most easily constructed by ``restricting'' model structures from $\ch$ to the category $\chain{\class{F}}$ of flat modules.

So let $R$ be a ring and denote the class of flat modules by $\class{F}$. The modules in $\class{C} = \rightperp{\class{F}}$ are called cotorsion modules. A important result from~\cite{enochs-flat-cover-theorem} is that $(\class{F},\class{C})$ is a complete cotorsion pair. Being a cotorsion pair, $\class{F}$ is closed under extensions and so it inherits the structure of an exact category where the short exact sequences are the usual ones but with all three terms in $\class{F}$. Since $\class{F}$ is closed under retracts this is automatically a WIC exact structure, and since it is closed under coproducts it is automatically IC exact by comments in Section~\ref{subsec-chain complexes form a Frobenius cat}. We let $\chain{\class{F}}$ denote the category of all chain complexes of flat modules along with the degreewise exact structure coming from $\class{F}$, and we conclude that $\chain{\class{F}}$ is also IC exact from Lemma~\ref{lemma-ch(A)}. We also have from Proposition~\ref{prop-Ch(A)_dw is Frobenius}, the IC Frobenius category $\cha{F}_{dw}$, which has the degreewise split structure. Following previous notation of the author we have the following classes of chain complexes.
\begin{itemize}
\item $\dwclass{F}$ = Class of all complexes of flat modules.
\item $\exclass{F}$ = Class of all exact complexes of flat modules.
\item $\tilclass{F}$ = Class of all exact complexes of flat modules having flat cycle modules.
\end{itemize}
Note that $\tilclass{F}$ is precisely the class of all pure exact complexes of flat modules.
We use similar notation for complexes of projectives, for example, $\dwclass{P}$ will denote the class of all complexes of projectives. To be clear, we point out that we are using the notation $\dwclass{F}$ to denote just the class of all complexes of flat modules, without any extra structure, while $\chain{\class{F}}$ denotes the category with the exact structure.

Interestingly enough, we see two natural contenders for the derived category $\class{D}(\class{F})$. The most obvious choice is $\cha{F}/\exclass{F}$ because $\exclass{F}$ is the class of exact complexes of flat modules. As we will see in this section, this category is equivalent to the usual derived category $\class{D}(R)$. Generally speaking, if $(\class{F},\class{C})$ is a complete hereditary cotorsion pair in $\ch$, then the ``ambient'' derived category $\cha{F}/\exclass{F}$ will be equivalent to the usual derived category, basically because $\cha{F}$ contains the DG-projective complexes. On the other hand, following the definition in~\cite{buhler-exact categories}, an exact complex in the exact category $\cha{F}$ boils down to the complexes in $\tilclass{F}$. So it is $\cha{F}/\tilclass{F}$ that we will refer to as the derived category of flat modules, and this corresponds to Murfet's \emph{mock homotopy category of projectives}~\cite{murfet}. The notation such as $\cha{F}/\tilclass{F}$ makes sense once we establish that we have an exact model structure on $\cha{F}$ having $\tilclass{F}$ as the trivial objects. It then follows from Proposition~\ref{prop-triangulated localization} that $\cha{F}/\tilclass{F}$ is universal with respect to ``killing'' the complexes in $\tilclass{F}$. That is, for any exact functor $\cha{F} \xrightarrow{} \class{T}$ where $\class{T}$ is triangulated, if the functor takes complexes in $\tilclass{F}$ to 0, the functor uniquely extends to the homotopy category $\cha{F}/\tilclass{F}$.

The following lemma makes clear just what exactly the projective and injective objects are in $\cha{F}$. Although we are only working with the cotorsion pair $(\class{F},\class{C})$ in this section, note that the lemma holds more generally for other hereditary cotorsion pairs.

\begin{lemma}\label{lemma-classification of projectives and injectives in Ch(F)}
The exact category $\class{F}$ has enough projectives and enough injectives. The projectives are the usual projective modules and the injectives are the cotorsion flat modules. This gives us the following.
\begin{enumerate}
\item $\cha{F}$ has enough projectives and the following are equivalent:
\begin{itemize}
\item $X$ is projective in $\cha{F}$.
\item $X$ is a projective chain complex in $\ch$. That it, $X$ is exact with projective cycle modules.
\item $X$ is a contractible complex with projective components.
\item $X$ is a split exact complex with projective components.
\end{itemize}
\item $\cha{F}$ has enough injectives and the following are equivalent:
\begin{itemize}
\item $X$ is injective in $\cha{F}$.
\item $X$ is a cotorsion flat chain complex in $\ch$. That it, $X$ is exact with cotorsion flat cycle modules.
\item $X$ is a contractible complex with cotorsion flat components.
\item $X$ is a split exact complex with cotorsion flat components.
\end{itemize}
\end{enumerate}
\begin{enumerate}
\setcounter{enumi}{2}
\item $\cha{F}_{dw}$ is Frobenius and the following are equivalent:
\begin{itemize}
\item $X$ is projective-injective in $\cha{F}_{dw}$.
\item $X$ is a contractible complex with flat components.
\item $X$ is a split exact complex with flat components.
\end{itemize}
\end{enumerate}
\end{lemma}

\begin{proof}
Since $\class{F}$ is idempotent complete, this comes from Proposition~\ref{prop-Ch(A)_dw is Frobenius} and Corollary~\ref{cor-Ch(A) enough projectives} once we see that the projective and injective objects in $\class{F}$ are as claimed. But this is easy. The projectives in this category are the usual projectives, for if $F$ is projective in $\class{F}$, then write an epimorphism $P \twoheadrightarrow F$ with $P$ a projective module. This epimorphism must split making $F$ projective too. We claim that the injectives in $\class{F}$ are the cotorsion flat modules. To see this, note that if $F \in \class{F}$ is injective we can use completeness of the flat cotorsion pair to find a short exact sequence $0 \xrightarrow{} F \xrightarrow{} C \xrightarrow{} F' \xrightarrow{} 0$ with $C$ cotorsion flat and $F'$ flat. This is an admissible s.e.s. in $\class{F}$ and so must split, making $F$ cotorsion flat.
\end{proof}

\begin{proposition}\label{prop-generating projective cotorsion pairs in Ch(F)}
 Let $(\class{P},\class{W})$ be a projective cotorsion pair in $\ch$ with $\class{P}$ contained in the class of degreewise projective complexes. Then $(\class{P},\class{W} \cap \dwclass{F})$ is a projective cotorsion pair in the exact category $\chain{\class{F}}$. The class $\class{P}$ cogenerates a localizing cotorsion pair $(K(\class{P}),\class{W} \cap \dwclass{F})$ in the Frobenius category $\chain{\class{F}}_{dw}$ and the class $K(\class{P})$ is characterized by the following equivalent statements for a complex $X \in \chain{\class{F}}$:
\begin{enumerate}
\item $X \in K(\class{P})$
\item There exists a complex $P \in \class{P}$ and contractible complexes of flats $W_1, W_2 \in \chain{\class{F}}$ such that $X \oplus W_1 \cong P \oplus W_2$.
\item There exists a complex $P \in \class{P}$ and contractible complexes $W_1, W_2 \in \ch$ such that $X \oplus W_1 \cong P \oplus W_2$.
\item $X$ is chain homotopy equivalent to some complex $P \in \class{P}$.
\end{enumerate}
\end{proposition}

\begin{proof}
Lets see first that $(\class{P},\class{W} \cap \dwclass{F})$ is a projective cotorsion pair in $\chain{\class{F}}$. First we note that $\rightperp{\class{P}} = \class{W} \cap \dwclass{F}$, where the ``perp'' here is taken in $\chain{\class{F}}$. That is, starting with an $X \in \chain{\class{F}}$ we have $\Ext^1(P,X)=0$ for all $P \in \class{P}$ if and only if $X \in \class{W}$. So taking the ``left-perp'', inside $\chain{\class{F}}$ again, we automatically have $\class{P} \subseteq \leftperp{[\class{W} \cap \dwclass{F}]}$. So to see that $(\class{P},\class{W} \cap \dwclass{F})$ is a cotorsion pair in $\chain{\class{F}}$ we need to show $\leftperp{[\class{W} \cap \dwclass{F}]} \subseteq \class{P}$.
Before proceeding we point out the following: For an arbitrary complex of flats $X$, since $(\class{P},\class{W})$ is a complete cotorsion pair in $\ch$ we can find a short exact sequence $0 \xrightarrow{} W \xrightarrow{} P \xrightarrow{} X \xrightarrow{} 0$ with $P \in \class{P}$ and $W \in \class{W} \cap \dwclass{F}$. Indeed each $W_n$ is flat since $(\class{F},\class{C})$ is an hereditary cotorsion pair. In particular this is an admissible short exact sequence in $\chain{\class{F}}$. Using this observation we now show $\leftperp{[\class{W} \cap \dwclass{F}]} \subseteq \class{P}$.
So let $X \in \leftperp{[\class{W} \cap \dwclass{F}]}$, where of course we mean $X$ is in $\chain{\class{F}}$ and this ``left-perp'' is taken in $\chain{\class{F}}$. Then the admissible s.e.s. $0 \xrightarrow{} W \xrightarrow{} P \xrightarrow{} X \xrightarrow{} 0$ must split, telling us $X$ is a retract of $P \in \class{P}$. It follows that $X$ too is in $\class{P}$. This shows that $(\class{P}, \class{W} \cap \dwclass{F})$ is a cotorsion pair in $\chain{\class{F}}$ and we observe that it is virtually automatic that $(\class{P}, \class{W} \cap \dwclass{F})$ is a projective cotorsion pair in $\chain{\class{F}}$. Indeed the cotorsion pair has enough projectives due to the observation of the existence of the admissible short exact sequence above. In the same way, it is obvious that the cotorsion pair has enough injectives, since $(\class{P}, \class{W})$ does in $\ch$. Since by hypothesis $\class{W}$ is thick and contains the projective complexes, we see that $\class{W} \cap \dwclass{F}$ is thick in $\chain{\class{F}}$ and contains the projective complexes, which are the projective objects in $\chain{\class{F}}$ by Lemma~\ref{lemma-classification of projectives and injectives in Ch(F)}. So we see from~\cite[Proposition~3.7~(3)]{gillespie-recollements} that $(\class{P}, \class{W} \cap \dwclass{F})$ is a projective cotorsion pair in $\chain{\class{F}}$.

The rest is automatic from the dual of Theorem~\ref{them-generalized K-injective models} and its continuation in Proposition~\ref{prop-characterization of KF-injective complexes}. We will just add a reason as to why (2) and (3) are equivalent. As stated in the proof of Prop~\ref{prop-characterization of KF-injective complexes}, in any WIC Frobenius category $\cat{A}$, we know that objects $A$ and $B$ are isomorphic in the stable category $\cat{A}/\sim$ if and only if there are projective-injective objects $W_1,W_2$ in $\cat{A}$ and an $\cat{A}$-isomorphism $A \oplus W_1 \cong B \oplus W_2$. Applying this to the WIC Frobenius $\cha{F}_{dw}$ gives us (2) if and only if (4) since according to Lemma~\ref{lemma-classification of projectives and injectives in Ch(F)} the stable category of $\chain{\class{F}}_{dw}$ is the usual homotopy category and the projective-injective objects are the split exact complexes of with flat components. But applying the same idea to the WIC Frobenius $\ch_{dw}$ gives us (3) if and only if (4) as well.
\end{proof}

\begin{remark}
We note that the identity functor on $\chain{\class{F}}$ is a (left) Quillen equivalence from the projective model structure $(\class{P},\class{W} \cap \dwclass{F})$ on $\chain{\class{F}}$ to the projective model structure induced by $(K(\class{P}),\class{W} \cap \dwclass{F})$. Moreover, the homotopy categories associated to these model structures are equivalent to the usual homotopy categories of complexes $\class{P}/\sim$ and $K(\class{P})/\sim \,$.
\end{remark}

We now turn around and look at a similar way to construct injective cotorsion pairs in $\chain{\class{F}}$.

\begin{proposition}\label{prop-generating injective cotorsion pairs in Ch(F)}
Let $(\hat{\class{F}}, \hat{\class{C}})$ be a complete cotorsion pair in $\ch$ with $\hat{\class{F}} \subseteq \dwclass{F}$. Assume that $\hat{\class{F}}$ is thick in the exact category $\chain{\class{F}}$ and that $\hat{\class{F}}$ contains all the split exact complexes with flat components. Then $(\hat{\class{F}}, \hat{\class{C}} \cap \dwclass{F})$ is an injective cotorsion pair in $\chain{\class{F}}$. The class $\hat{\class{C}} \cap \dwclass{F}$ generates a localizing cotorsion pair $(\hat{\class{F}}, K(\hat{\class{C}} \cap \dwclass{F}))$ in the Frobenius category $\chain{\class{F}}_{dw}$ and the class $K(\hat{\class{C}} \cap \dwclass{F})$ is characterized by the following equivalent statements for a complex $X \in \chain{\class{F}}$:
\begin{enumerate}
\item $X \in K(\hat{\class{C}} \cap \dwclass{F})$
\item There exists a $C \in \hat{\class{C}} \cap \dwclass{F}$ and contractible complexes of flats $W_1, W_2 \in \chain{\class{F}}$ such that $X \oplus W_1 \cong C \oplus W_2$.
\item There exists a $C \in \hat{\class{C}} \cap \dwclass{F}$ and contractible complexes $W_1, W_2 \in \ch$ such that $X \oplus W_1 \cong C \oplus W_2$.
\item $X$ is chain homotopy equivalent to some $C \in \hat{\class{C}} \cap \dwclass{F}$.
\end{enumerate}
\end{proposition}

\begin{proof}
We first show that $(\hat{\class{F}}, \hat{\class{C}} \cap \dwclass{F})$ is a cotorsion pair in $\chain{\class{F}}$. Since $\rightperp{\hat{\class{F}}} = \hat{\class{C}}$ in $\ch$, it is clear that in $\chain{\class{F}}$  we have $\rightperp{\hat{\class{F}}} = \hat{\class{C}} \cap \dwclass{F}$. That is, starting with an $X \in \chain{\class{F}}$ we have $\Ext^1(F,X)=0$ for all $F \in \hat{\class{F}}$ if and only if $X \in \hat{\class{C}}$. Next, taking the ``left-perp'' inside $\chain{\class{F}}$ we automatically have $\hat{\class{F}} \subseteq \leftperp{[\hat{\class{C}} \cap \dwclass{F}]}$, so it is left to show $\leftperp{[\hat{\class{C}} \cap \dwclass{F}]} \subseteq \hat{\class{F}}$. So let $X \in \cha{F}$ be in $\leftperp{[\hat{\class{C}} \cap \dwclass{F}]}$. Since $(\hat{\class{F}},\hat{\class{C}})$ is a complete cotorsion pair in $\ch$ we can find a short exact sequence $0 \xrightarrow{} C \xrightarrow{} F \xrightarrow{} X \xrightarrow{} 0$ with $F \in \hat{\class{F}}$ and $C \in \hat{\class{C}}$. Since each $F_n,X_n$ are flat, each $C_n$ must also be flat. So $C \in \hat{\class{C}} \cap \dwclass{F}$. Since each $C_n$ is cotorsion and each $X_n$ is flat it is a degreewise split sequence and so it represents an admissible short exact sequence in both $\chain{\class{F}}$ and $\chain{\class{F}}_{dw}$. Anyway, it must split by hypothesis, telling us $X$ is a retract of $F \in \hat{\class{F}}$. It follows that $X$ too is in $\hat{\class{F}}$. This shows that $(\hat{\class{F}}, \hat{\class{C}} \cap \dwclass{F})$ is a cotorsion pair in $\chain{\class{F}}$ and we observe that it is virtually automatic that $(\hat{\class{F}}, \hat{\class{C}} \cap \dwclass{F})$ is an injective cotorsion pair in $\chain{\class{F}}$. Indeed the cotorsion pair has enough projectives due to the observation above. In the same way, it is obvious that the cotorsion pair has enough injectives, since $(\hat{\class{F}}, \hat{\class{C}})$ does in $\ch$. By assumption $\hat{\class{F}}$ is thick in $\chain{\class{F}}$ and contains the injective objects, which are the split exact complexes of cotorsion flat modules. So we see from~\cite[Proposition~3.6~(3)]{gillespie-recollements} that $(\hat{\class{F}}, \hat{\class{C}} \cap \dwclass{F})$ is an injective cotorsion pair in $\chain{\class{F}}$.

The rest is now automatic from Theorem~\ref{them-generalized K-injective models} and its continuation in Proposition~\ref{prop-characterization of KF-injective complexes}.
\end{proof}

\begin{remark}
The identity functor on $\chain{\class{F}}$ is a (right) Quillen equivalence from the injective model structure $(\hat{\class{F}}, \hat{\class{C}} \cap \dwclass{F})$ on $\chain{\class{F}}$ to the injective model structure induced by $(\hat{\class{F}}, K(\hat{\class{C}} \cap \dwclass{F}))$. Moreover, the homotopy categories associated to these model structures are equivalent to the usual homotopy categories of complexes $(\hat{\class{C}} \cap \dwclass{F})/\sim$ and $K(\hat{\class{C}} \cap \dwclass{F})/\sim \,$.
\end{remark}

%%%%%%%%%%%%%%%%%%%%%%%%%%%%%%%%%%%%%%%%%%%%%%%%%%%%%%%%
\subsection{The derived category of complexes of flat modules}

We now interpret work of Neeman and Murfet in terms of exact model structures. In particular, we construct both a projective and an injective model structure on $\cha{F}$ whose homotopy category is the derived category $\cha{F}/\tilclass{F}$ discussed in the introduction. Recall that $\tilclass{F}$ denotes the class of exact complexes in the exact category $\cha{F}$, so these are the exact complexes with each cycle module flat. From basic facts of purity, it is clear that these are precisely the pure exact complexes which lie in $\cha{F}$. We also recall now  that $\dwclass{P}$ denotes the class of all complexes of projective modules. An important technical fact proved in~\cite{neeman} is that $\rightperp{\dwclass{P}} \cap \dwclass{F} = \tilclass{F}$. Using this, our results follow quickly as corollaries to the above propositions because we already have a couple of nicely behaved cotorsion pairs.

\begin{corollary}\label{cor-the projective cotorsion pair for the mock homotopy category of projectives}
$(\dwclass{P}, \tilclass{F})$ is a projective cotorsion pair in the exact category $\chain{\class{F}}$. The class $\dwclass{P}$ cogenerates a localizing cotorsion pair $(K(\dwclass{P}), \tilclass{F})$ in the Frobenius category $\chain{\class{F}}_{dw}$ and the class $K(\dwclass{P})$ is characterized by the following equivalent statements for a complex $X \in \chain{\class{F}}$:
\begin{enumerate}
\item $X \in K(\dwclass{P})$
\item There exists a complex of projectives $P \in \dwclass{P}$ and contractible complexes of flats $W_1, W_2 \in \chain{\class{F}}$ such that $X \oplus W_1 \cong P \oplus W_2$.
\item There exists a complex of projectives $P \in \dwclass{P}$ and contractible complexes $W_1, W_2 \in \ch$ such that $X \oplus W_1 \cong P \oplus W_2$.
\item $X$ is chain homotopy equivalent to some complex of projectives $P$.
\end{enumerate}
Moreover, the identity functor on $\chain{\class{F}}$ is a (left) Quillen equivalence from the projective model structure $(\dwclass{P}, \tilclass{F})$ on $\chain{\class{F}}$ to the projective model structure induced by $(K(\dwclass{P}), \tilclass{F})$.
\end{corollary}

\begin{proof}
We know from \cite[Proposition~7.3~(1)]{gillespie-recollements} that $(\dwclass{P}, \rightperp{\dwclass{P}})$ is a projective cotorsion pair in $\ch$. (Note: In particular this means $\rightperp{\dwclass{P}} \cap \dwclass{P}$ is precisely the class of all projective complexes while Neeman's result says that $\rightperp{\dwclass{P}} \cap \dwclass{F}$ is precisely the class of all flat complexes.) So by Proposition~\ref{prop-generating projective cotorsion pairs in Ch(F)} we know $(\dwclass{P}, \rightperp{\dwclass{P}} \cap \dwclass{F}) = (\dwclass{P} , \tilclass{F})$
is a projective cotorsion pair in $\chain{\class{F}}$. The remaining statements also follow from Proposition~\ref{prop-generating projective cotorsion pairs in Ch(F)} and the Remark that follows that proposition.

\end{proof}

We note again that by establishing $(\dwclass{P}, \tilclass{F})$ as a projective cotorsion pair on $\chain{\class{F}}$ we automatically have that the triangulated localization $\cha{F}/\tilclass{F}$ exists. Moreover we get an equivalence of triangulated categories $\cha{F}/\tilclass{F} \cong \dwclass{P}/\sim$ where $\sim$ is the usual relation of chain homotopy. So we think of $(\dwclass{P}, \tilclass{F})$ as a projective model for Murfet's \emph{mock homotopy category of projectives}. Of course this model structure doesn't generalize to sheaf categories which don't have enough projectives. So the original point of Murfet's work was to describe this category without using projectives. We see now that $\cha{F}/\tilclass{F}$ has an injective model structure with the same trivial objects $\tilclass{F}$. This is analogous to how $\class{D}(R)$ has the standard projective model $(\dgclass{P}, \class{E})$ and the standard injective model $(\class{E},\dgclass{I})$, sharing the class $\class{E}$ of exact complexes as the trivial objects. However, to get this balance on $\cha{F}$, the fibrant objects for our injective model structure for $\cha{F}/\tilclass{F}$ will not be build from injective modules, but rather the cotorsion flat modules. In particular, recall that in the notation and language of~\cite{gillespie} we have the flat cotorsion pair $(\tilclass{F},\dgclass{C})$ in $\ch$. Here we will call a complex $X \in \dgclass{C}$ a DG-cotorsion complex. The cotorsion pair is complete, see for example~\cite[Corollary~4.10]{gillespie}. The class $\dgclass{C}$ consists precisely of all the complexes $Y$ for which $Y_n$ is cotorsion and any $f : F \xrightarrow{} Y$ is null homotopic whenever $F \in \tilclass{F}$.

\begin{corollary}\label{cor-the injective cotorsion pair for the mock homotopy category of projectives}
$(\tilclass{F}, \dgclass{C} \cap \dwclass{F})$ is an injective cotorsion pair in the exact category $\chain{\class{F}}$. The class $\dgclass{C} \cap \dwclass{F}$ generates a localizing cotorsion pair $(\tilclass{F}, K(\dgclass{C} \cap \dwclass{F}))$ in the Frobenius category $\chain{\class{F}}_{dw}$ and the class $K(\dgclass{C} \cap \dwclass{F})$ is characterized by the following equivalent statements for a complex $X \in \chain{\class{F}}$:
\begin{enumerate}
\item $X \in K(\dgclass{C} \cap \dwclass{F})$
\item There exists a DG-cotorsion complex of flats $C \in \dgclass{C} \cap \dwclass{F}$ and contractible complexes of flats $W_1, W_2 \in \chain{\class{F}}$ such that $X \oplus W_1 \cong C \oplus W_2$.
\item There exists a DG-cotorsion complex of flats $C \in \dgclass{C} \cap \dwclass{F}$ and contractible complexes $W_1, W_2 \in \ch$ such that $X \oplus W_1 \cong C \oplus W_2$.
\item $X$ is chain homotopy equivalent to some DG-cotorsion complex of flats $C$.
\end{enumerate}
Moreover, the identity functor on $\chain{\class{F}}$ is a (right) Quillen equivalence from the injective model structure $(\tilclass{F}, \dgclass{C} \cap \dwclass{F})$ on $\chain{\class{F}}$ to the injective model structure induced by $(\tilclass{F}, K(\dgclass{C} \cap \dwclass{F}))$.
\end{corollary}

\begin{proof}
It is easy to see that $\tilclass{F}$ is thick in $\cha{F}$ and that it contains the contractible complexes of flats. So the result follows from~\ref{prop-generating injective cotorsion pairs in Ch(F)}.
\end{proof}

Having this model structure leads us immediately to the conclusion $\cha{F}/\tilclass{F} \cong (\dgclass{C} \cap \dwclass{F})/\sim \,$ where $\sim$ is the usual relation of chain homotopy. Note also that we now have shown that $(K(\dwclass{P}),\, \tilclass{F} ,\, K(\dgclass{C} \cap \dwclass{F}))$ is a localizing cotorsion triple in the Frobenius category $\chain{\class{F}}_{dw}$. Denoting the stable category of $\cha{F}_{dw}$ by $K(\class{F})$, we summarize these results below including the recollement of Neeman from~\cite{neeman}. Again, the recollement only exists in the affine case.

\begin{corollary}\label{cor-neemans recollement}
$(K(\dwclass{P}),\, \tilclass{F} ,\, K(\dgclass{C} \cap \dwclass{F}))$ is a localizing cotorsion triple in $\chain{\class{F}}_{dw}$. Corollary~\ref{cor-recollements for WIC Frobenius cats} tells us we have equivalences of triangulated categories
$$K(\dwclass{P})/\sim \ \cong \ \cha{F}/\tilclass{F} \ \cong \ K(\dgclass{C} \cap \dwclass{F})/\sim $$ and $\sim$ is the usual chain homotopy. Setting $\class{X} = K(\dwclass{P})$ and $\class{Z} = K(\dgclass{C} \cap \dwclass{F})$ to simplify notation
we also have a recollement as shown.
\[
\begin{tikzpicture}[node distance=3.5cm]
\node (A) {$\tilclass{F}/\sim$};
\node (B) [right of=A] {$K(\class{F})$};
\node (C) [right of=B] {$\cha{F}/\tilclass{F}$};
\draw[<-,bend left=40] (A.20) to node[above]{\small E$(\class{X},\tilclass{F})$} (B.160);
\draw[->] (A) to node[above]{\small $I$} (B);
\draw[<-,bend right=40] (A.340) to node [below]{\small C$(\tilclass{F},\class{Z})$} (B.200);
\draw[<-,bend left] (B.20) to node[above]{\small $\lambda = \text{C}(\class{X},\tilclass{F})$} (C.160);
\draw[->] (B) to node[above]{\small Q} (C);
\draw[<-,bend right] (B.340) to node [below]{\small $\rho = \text{E}(\tilclass{F},\class{Z})$} (C.200);
\end{tikzpicture}
\]

\end{corollary}

\begin{proof}
This follows from Corollaries~\ref{cor-the projective cotorsion pair for the mock homotopy category of projectives} and~\ref{cor-the injective cotorsion pair for the mock homotopy category of projectives} and~\ref{cor-recollements for WIC Frobenius cats}.

\end{proof}

%%%%%%%%%%%%%%%%%%%%%%%%%%%%%%%%%%%%%%%%%%%%%%%%%%%%%%%%
\subsection{The category of flat complexes modulo the exact complexes}

Above we found both a projective and an injective model structure for $\cha{F}/\tilclass{F}$ and recovered the recollement of Neeman. As mentioned at the beginning of this section, one might also consider the category $\cha{F}/\exclass{F}$. We now see that we again get both a projective and an injective model structure for this triangulated category and get a similar recollement to the one of Neeman in Corollary~\ref{cor-neemans recollement}. It is really a version of Verdier's recollement from Corollary~\ref{cor-Verdier recollement via cotorsion pairs} but for the category $\cha{F}$. Indeed it turns out that $\cha{F}/\exclass{F} \cong \class{D}(R)$.

The projective model structure for $\cha{F}/\tilclass{F}$ was obtained by restricting the cotorsion pair $(\dwclass{P},\rightperp{\dwclass{P}})$ in $\ch$ to $\cha{F}$.
For $\cha{F}/\exclass{F}$, we now start by restricting the well-known cotorsion pair $(\dgclass{P},\class{E})$ where $\dgclass{P}$ are the DG-projective complexes and $\class{E}$ are the exact complexes, again to the category $\cha{F}$.

\begin{corollary}\label{cor-the projective cotorsion pair for the derived category of flats}
$(\dgclass{P}, \exclass{F})$ is a projective cotorsion pair in the exact category $\chain{\class{F}}$. The class $\dgclass{P}$ cogenerates a localizing cotorsion pair $(K(\dgclass{P}), \exclass{F})$ in the Frobenius category $\chain{\class{F}}_{dw}$ and $K(\dgclass{P})$ is precisely class of K-projective complexes of flat modules. Moreover, the identity functor on $\chain{\class{F}}$ is a (left) Quillen equivalence from the projective model structure $(\dgclass{P}, \exclass{F})$ on $\chain{\class{F}}$ to the projective model structure induced by $(K(\dgclass{P}), \exclass{F})$.
\end{corollary}

\begin{proof}
Since $(\dgclass{P}, \class{E})$ is a projective cotorsion pair, we know from Proposition~\ref{prop-generating projective cotorsion pairs in Ch(F)} that $(\dgclass{P}, \class{E} \cap \dwclass{F}) = (\dgclass{P} , \exclass{F})$ is a projective cotorsion pair in $\chain{\class{F}}$. The remaining statements follow from Proposition~\ref{prop-generating projective cotorsion pairs in Ch(F)} and the Remark that follows it. In particular, note that $K(\dgclass{P})$ consists precisely of the complexes of flats which are chain homotopy equivalent to a DG-projective complex. That is, $K(\dgclass{P})$ is the class of K-projective complexes of flats.

\end{proof}

We now set $\class{C} = \rightperp{\exclass{F}}$ in $\ch$. It was shown in\cite[Theorem~5.5]{gillespie-degreewise-model-strucs} that $(\exclass{F}, \class{C})$ is a complete cotorsion pair in $\ch$ and $\class{C}$ is precisely the class of all complexes $C$ of cotorsion modules for which any $f : F \xrightarrow{} C$ with $F \in \exclass{F}$ is null homotopic.

\begin{corollary}\label{cor-the injective cotorsion pair for the derived category of flats}
$(\exclass{F}, \class{C} \cap \dwclass{F})$ is an injective cotorsion pair in the exact category $\chain{\class{F}}$. The class $\class{C} \cap \dwclass{F}$ generates a localizing cotorsion pair $(\exclass{F}, K(\class{C} \cap \dwclass{F}))$ in the Frobenius category $\chain{\class{F}}_{dw}$ and the class $K(\class{C} \cap \dwclass{F})$ is characterized by the following equivalent statements for a complex $X \in \chain{\class{F}}$:
\begin{enumerate}
\item $X \in K(\class{C} \cap \dwclass{F})$.
\item There exists a complex $C \in \class{C} \cap \dwclass{F}$ and contractible complexes of flats $W_1, W_2 \in \chain{\class{F}}$ such that $X \oplus W_1 \cong C \oplus W_2$.
\item There exists a complex $C \in \class{C} \cap \dwclass{F}$ and contractible complexes $W_1, W_2 \in \ch$ such that $X \oplus W_1 \cong C \oplus W_2$.
\item $X$ is chain homotopy equivalent to some $C \in \class{C} \cap \dwclass{F}$.
\end{enumerate}
Moreover, the identity functor on $\chain{\class{F}}$ is a (right) Quillen equivalence from the injective model structure $(\exclass{F}, \class{C} \cap \dwclass{F})$ on $\chain{\class{F}}$ to the injective model structure induced by $(\exclass{F}, K(\class{C} \cap \dwclass{F}))$.
\end{corollary}

\begin{proof}
It is easy to see that $\exclass{F}$ is thick in $\cha{F}$ and that it contains the contractible complexes of flats. So the result follows from Proposition~\ref{prop-generating injective cotorsion pairs in Ch(F)} and the fact mentioned above that $(\exclass{F}, \class{C})$ is already known to be a complete cotorsion pair in $\ch$.

\end{proof}

The following corollary is now automatic from our work in Section~\ref{sec-localizing cot pairs in frobenius categories}.

\begin{theorem}\label{them-my recollement for derived category of flats}
$(K(\dgclass{P}),\, \exclass{F} ,\, K(\class{C} \cap \dwclass{F}))$ is a localizing cotorsion triple in $\chain{\class{F}}_{dw}$. Corollary~\ref{cor-recollements for WIC Frobenius cats} tells us we have equivalences of triangulated categories
$$K(\dgclass{P})/\sim \ \cong \ \cha{F}/\exclass{F} \ \cong \ K(\class{C} \cap \dwclass{F})/\sim $$ and $\sim$ is the usual chain homotopy.
Setting $\class{X} = K(\dgclass{P})$ and $\class{Z} = K(\class{C} \cap \dwclass{F})$ to simplify notation
we also have a recollement as shown.
\[
\begin{tikzpicture}[node distance=3.5cm]
\node (A) {$\exclass{F}/\sim$};
\node (B) [right of=A] {$K(\class{F})$};
\node (C) [right of=B] {$\cha{F}/\exclass{F}$};
\draw[<-,bend left=40] (A.20) to node[above]{\small E$(\class{X},\exclass{F})$} (B.160);
\draw[->] (A) to node[above]{\small $I$} (B);
\draw[<-,bend right=40] (A.340) to node [below]{\small C$(\exclass{F},\class{Z})$} (B.200);
\draw[<-,bend left] (B.20) to node[above]{\small $\lambda = \text{C}(\class{X},\exclass{F})$} (C.160);
\draw[->] (B) to node[above]{\small Q} (C);
\draw[<-,bend right] (B.340) to node [below]{\small $\rho = \text{E}(\exclass{F},\class{Z})$} (C.200);
\end{tikzpicture}
\]

\end{theorem}

\begin{proof}
This follows from Corollaries~\ref{cor-the projective cotorsion pair for the derived category of flats} and~\ref{cor-the injective cotorsion pair for the derived category of flats} and~\ref{cor-recollements for WIC Frobenius cats}.

\end{proof}

The following corollary is interesting and as indicated at the start of this section is true in more general situations.

\begin{corollary}\label{cor-derived cat of flat is equivalent to usual D(R)}
For any ring $R$, we have an equivalence $\cha{F}/\exclass{F} \ \cong \class{D}(R)$.
\end{corollary}

\begin{proof}
Each is equivalent to $\dgclass{P}/\sim \,$, the homotopy category of DG-projective complexes.
\end{proof}

\subsection{The mock stable derived category and recollement}

There is a recollement due to Murfet in~\cite{murfet} which expresses $\cha{F}/\tilclass{F}$ by adjoining to the derived category of $R$, the \emph{mock stable derived category} of $R$, which is the subcategory of $\cha{F}/\tilclass{F}$ consisting of the exact complexes. We now interpret this through the recollement Theorem~\ref{them-recollements in krause form} and the model structures we have constructed.  However, we first must construct the model structures corresponding to Murfet's mock stable derived category. Since we are working here with modules over a ring $R$, we wish to construct both a projective and an injective model structure and then to obtain both a projective version and an injective version of the recollement. In general, it is the injective version only that works for sheaves.

First, the projective case. We already have used the fact that $(\dwclass{P}, \class{W}_1)$ where $\class{W}_1 = \rightperp{\dwclass{P}}$ is a projective cotorsion pair in $\ch$. It is also known that $(\exclass{P}, \class{W}_2)$ is a projective cotorsion pair in $\ch$; see~\cite[Proposition~7.3]{gillespie-recollements}. Here $\exclass{P}$ are the exact complexes of projectives and $\class{W}_2 = \rightperp{\exclass{P}}$. So Proposition~\ref{prop-generating projective cotorsion pairs in Ch(F)} gives a corresponding projective cotorsion pair $(\exclass{P}, \class{W}_2 \cap \dwclass{F})$  in $\cha{F}$. This model corresponds to Murfet's mock stable derived category. Of course, here in the affine case we have enough projectives and so we are not ``mocking'' anything!

\begin{corollary}\label{cor-affine case of murfet recollement}
Set $\class{M}_1 = (\dwclass{P}, \tilclass{F})$ and $\class{M}_2 = (\exclass{P}, \class{W}_2 \cap \dwclass{F})$ and $\class{M}_3 = (\dgclass{P}, \exclass{F})$. These are each projective cotorsion pairs in $\chain{\class{F}}$ and clearly satisfy $\exclass{F} \cap \dwclass{P} = \exclass{P}$, so we automatically recover the recollement below from Theorem~\ref{them-recollements theorem projective version}.
\[
\begin{tikzpicture}[node distance=3.5cm]
\node (A) {$\exclass{P}/\sim$};
\node (B) [right of=A] {$\dwclass{P}/\sim$};
\node (C) [right of=B] {$\cha{F}/\exclass{F}$};
\draw[<-,bend left=40] (A.20) to node[above]{\small E$(\class{M}_3)$} (B.160);
\draw[->] (A) to node[above]{\small $I$} (B);
\draw[<-,bend right=40] (A.340) to node [below]{\small C$(\class{M}_2)$} (B.200);
\draw[<-,bend left] (B.20) to node[above]{\small $\lambda = \text{C}(\class{M}_3)$} (C.160);
\draw[->] (B) to node[above]{$Q$} (C);
\draw[<-,bend right] (B.340) to node [below]{\small $\rho = \text{E}(\class{M}_2) \circ \text{C}(\class{M}_1)$} (C.200);
\end{tikzpicture}
\]

\end{corollary}

Now we look at the injective version that avoids projectives. Here recall that $(\dgclass{F},\tilclass{C})$ is the DG-flat cotorsion pair from~\cite{gillespie}. The complexes in $\tilclass{C}$ are exact complexes with each cycle module a cotorsion module, and we will call these complexes \emph{cotorsion} complexes. The complexes in $\dgclass{F}$ are called \emph{DG-flat}. The cotorsion pair satisfies the hypotheses of Proposition~\ref{prop-generating injective cotorsion pairs in Ch(F)}, and so we have that $(\dgclass{F},\tilclass{C} \cap \dwclass{F})$ is an injective cotorsion pair in $\cha{F}$.

\begin{corollary}\label{cor-cotorsion-flat verions of murfet recollement}
Set $\class{M}_1 = (\tilclass{F}, \dgclass{C} \cap \dwclass{F})$ and $\class{M}_2 = (\dgclass{F}, \tilclass{C} \cap \dwclass{F})$ and $\class{M}_3 = (\exclass{F}, \class{C} \cap \dwclass{F})$. These are each injective cotorsion pairs in $\chain{\class{F}}$ and satisfy $\exclass{F} \cap (\dgclass{C} \cap \dwclass{F}) = \tilclass{C} \cap \dwclass{F}$, so we automatically recover the recollement below from Theorem~\ref{them-recollements in krause form}.
\[
\begin{tikzpicture}[node distance=3.5cm]
\node (A) {$(\tilclass{C} \cap \dwclass{F})/\sim$};
\node (B) [right of=A] {$(\dgclass{C} \cap \dwclass{F})/\sim$};
\node (C) [right of=B] {$\cha{F}/\exclass{F}$};
\draw[<-,bend left=40] (A.20) to node[above]{\small E$(\class{M}_2)$} (B.160);
\draw[->] (A) to node[above]{\small $I$} (B);
\draw[<-,bend right=40] (A.340) to node [below]{\small C$(\class{M}_3)$} (B.200);
\draw[<-,bend left] (B.20) to node[above]{\small $\lambda = \text{C}(\class{M}_2) \circ \text{E}(\class{M}_1)$} (C.160);
\draw[->] (B) to node[above]{$Q$} (C);
\draw[<-,bend right] (B.340) to node [below]{\small $\rho = \text{E}(\class{M}_2)$} (C.200);
\end{tikzpicture}
\]

\end{corollary}

\begin{proof}
We know from~\cite{gillespie} that $\class{E} \cap \dgclass{C} = \tilclass{C}$ where $\class{E}$ are the exact complexes, so it is clear that $\exclass{F} \cap (\dgclass{C} \cap \dwclass{F}) = \tilclass{C} \cap \dwclass{F}$.

\end{proof}

We end with a few remarks regarding the model structures constructed in this section.

\begin{remark}
Again, the model structures we've constructed on $\cha{F}$ coming from localizing cotorsion pairs in $\cha{F}_{dw}$ are not cofibrantly generated. However, the ones constructed using cotorsion pairs in $\cha{F}$ are. Indeed the exact category $\cha{F}$ is \emph{efficient} in the sense of~\cite[Definition~2.6]{saorin-stovicek}. As is nicely explained in~\cite[Section~2]{saorin-stovicek}, these are exact categories satisfying a few extra axioms allowing for Quillen's small object argument. In particular, the fact that our cotorsion pairs here are all cogenerated by a set allows for the construction of a set of generating cofibrations (resp. trivial cofibrations), using the methods of~\cite[Section~2]{saorin-stovicek}. In other words, the cotorsion pairs in $\cha{F}$ are \emph{small} in the sense of~\cite{hovey}.

\end{remark}

\begin{remark}
We end by pointing out that while the projective models constructed in this section don't generalize to sheaf categories, the injective ones do. So while we have worked only with modules over a ring $R$ in this section, Theorem~\ref{them-recollements in krause form} only requires that are working in a WIC exact category with enough injectives. When generalizing to categories of sheaves or quasi-coherent sheaves, the only technical results that need to be checked is completeness of the cotorsion pairs $(\tilclass{F},\dgclass{C})$,  $(\dgclass{F},\tilclass{C})$,  $(\dwclass{F},\rightperp{\dwclass{F}})$, and $(\exclass{F},\rightperp{\exclass{F}})$. If we assume $X$ is a scheme having a flat generator, such as when $X$ is quasi-compact and semi-separated, then these cotorsion pairs are complete for complexes of quasi-coherent sheaves. In this setting, the completeness and compatibility of $(\tilclass{F},\dgclass{C})$,  and $(\dgclass{F},\tilclass{C})$ was shown in~\cite{gillespie-quasi-coherent}, while the completeness and compatibility of  $(\dwclass{F},\rightperp{\dwclass{F}})$, and $(\exclass{F},\rightperp{\exclass{F}})$ was shown in~\cite[Theorem~5.5]{gillespie-degreewise-model-strucs}. Similarly, for the category of all sheaves on an arbitrary ringed space we refer to~\cite{gillespie-sheaves}. In fact, as described in~\cite{stovicek-Hill}, all of the true technicalities involved here follow from the fact that the class of all flat quasi-coherent sheaves is \emph{deconstructible}.

\end{remark}

\end{document}